\newif\iflong
\newif\ifshort

\longtrue

\iflong
\else
\shorttrue
\fi

\documentclass[11pt,]{article}

\usepackage[utf8]{inputenc}
\usepackage[T1]{fontenc}

\usepackage[british]{babel}

\usepackage{amssymb}
\usepackage{microtype}
\usepackage{amsmath}
\usepackage{amsthm}
\usepackage{amsfonts}
\usepackage{mathtools}
\usepackage{xcolor}

\usepackage{tikz}

\usepackage{algorithm,algorithmicx}
\usepackage{algpseudocode}

\usepackage[inline]{enumitem}
\usepackage{titling}
\usepackage{xspace}
\usepackage{dsfont}
\usepackage{stmaryrd}
\usepackage[hyphens]{url}
\usepackage{xspace}
\usepackage{graphicx}
\usepackage[colorinlistoftodos]{todonotes}
\usepackage[inline]{enumitem}
\setlist[itemize]{topsep=0pt,partopsep=0pt,itemsep=0pt,parsep=0pt}
\setlist[itemize,1]{label=\textbullet}
\setlist[itemize,2]{label=---}
\setlist[itemize,3]{label=*}
\setlist[enumerate]{topsep=0pt,partopsep=0pt,itemsep=0pt,parsep=0pt}
\setlist[enumerate,1]{label=(\roman*)}
\setlist[enumerate,2]{label=\alph*}
\setlist[enumerate,3]{label=\arabic*}
\usepackage{multicol}
\usepackage{xcolor}
\usepackage{amsthm}
\usepackage{thm-restate}
\usetikzlibrary{shapes.geometric,hobby,calc} 
\usetikzlibrary{decorations.pathmorphing}
\usetikzlibrary{decorations.text}
\usetikzlibrary{shapes.misc}
\usetikzlibrary{decorations,shapes,snakes}
\usepackage{etoolbox}
\usepackage{bbding}
\usepackage[subrefformat=parens]{subcaption} 
\usepackage{nicefrac}

\usepackage{hyperref}
\usepackage{cleveref}
\usepackage{macros}
\usepackage{stackengine}
\stackMath

\usepackage[b]{esvect}

\renewcommand{\tilde}{\widetilde}
{\begin{adjustwidth}{2em}{}\hspace{-2em}\textbf{Case #1.} }%
	{\end{adjustwidth}}

\colorlet{myGreen}{green!50!black}
\colorlet{myLightgreen}{green}
\colorlet{myRed}{red!90!black}
\definecolor{myBlue}{rgb}{0.25, 0.0, 1.0}
\definecolor{myLightBlue}{rgb}{0.39, 0.58, 0.93}
\colorlet{myViolet}{myBlue!55!myRed}
\definecolor{myOrange}{rgb}{1.0, 0.66, 0.07}

\hypersetup{
	colorlinks=true,
	linkcolor=myBlue,
	citecolor=myBlue,
	urlcolor=myLightBlue,
	bookmarksopen=true,
	bookmarksnumbered,
	bookmarksopenlevel=2,
	bookmarksdepth=3
}

\newcommand{\appsymb}{$\star$}
\newcommand{\appref}[1]{\ifshort{}[{\hyperref[proof:#1]{\appsymb}}]\fi{}}

\newif\ifcomment
\commentfalse
\commenttrue

\title{The Tight Cut Decomposition of Matching Covered \Uniformable Hypergraphs}
\predate{}
\date{}
\postdate{}

\preauthor{}
\DeclareRobustCommand{\authorthing}{
	\begin{center}
		\begin{tabular}{p{0.25\textwidth}p{.18\textwidth}p{.30\textwidth}}
			Isabel Beckenbach & Meike Hatzel & Sebastian Wiederrecht\\
			ZIB & TU Berlin & TU Berlin \\
			\emph{beckenbach@zib.de} & \emph{meike.hatzel@tu-berlin.de} & \emph{sebastian.wiederrecht@tu-berlin.de}
		\end{tabular}
\end{center}}
\author{\authorthing}
\postauthor{}

\usepackage[a4paper]{geometry}
\begin{document}
	
	\maketitle
	
	\begin{abstract}
		The perfect matching polytope, i.e.\@ the convex hull of (incidence vectors of) perfect matchings of a graph is used in many combinatorial algorithms.
		Kotzig, Lov{\'a}sz and Plummer developed a decomposition theory for graphs with perfect matchings and their corresponding polytopes known as the \emph{tight cut decomposition} which breaks down every graph into a number of indecomposable graphs, so called \emph{bricks}.
		For many properties that are of interest on graphs with perfect matchings, including the description of the perfect matching polytope, it suffices to consider these bricks.
		A key result by Lov{\'a}sz on the tight cut decomposition is that the list of bricks obtained is the same independent of the choice of tight cuts made during the tight cut decomposition procedure.
		This implies that finding a tight cut decomposition is polynomial time equivalent to finding a single tight cut.
		
		We generalise the notions of a tight cut, a tight cut contraction and a tight cut decomposition to hypergraphs.
		By providing an example, we show that the outcome of the tight cut decomposition on general hypergraphs is no longer unique.
		However, we are able to prove that the uniqueness of the tight cut decomposition is preserved on a slight generalisation of uniform hypergraphs.
		Moreover, we show how the tight cut decomposition leads to a decomposition of the perfect matching polytope of \uniformable hypergraphs.

	\end{abstract}
	

\newpage

\section{Introduction}


A \emph{perfect matching} of a graph $G$ is a subset $M\subseteq\Fkt{E}{G}$ of the edges of $G$,  such that for every vertex $v\in\Fkt{V}{G}$ there is exactly one edge in $M$ that contains $v$.
Similarly, for a hypergraph $H$ a perfect matching is a subset $M\subseteq\Fkt{E}{H}$ of the (hyper)edges of $H$, such that for every vertex $v\in\Fkt{V}{H}$ there is exactly one (hyper)edge in $M$ that contains $v$.

The structure of graphs with perfect matchings is a well established and active field of research (see for example \cite{de2014tight,wang2018core,lucchesi2017two}) and a main tool for the investigation of those graphs is a decomposition theory based on \emph{tight cuts}.

A \emph{cut} in a graph $G$ is a set of edges $C$ such that there exists a two-partition of $\Fkt{V}{G}$ and $C$ is exactly the set of edges with endpoints in both partitions, which are called the \emph{shores} of $C$.
Kotzig, Lov{\'a}sz and Plummer discovered that special kind of cuts, called \emph{tight cuts}, can be used to decompose the graph $G$ into smaller graphs while maintaining important information on the perfect matching polytope of the original graph (see \cite{lovasz2009matching}).
It is possible to decompose the smaller graphs obtained from a graph $G$ by applying a \emph{tight cut contraction}, i.e.\ identifying one of the shores of the tight cut as a single vertex and deleting multiple edges, even further until no tight cut can be found and thus a decomposition of $G$ into no longer decomposable graphs is obtained.
These graphs are called \emph{bricks}, or \emph{bricks and braces} if one wishes to distinguish between the non-bipartite and the bipartite case. 
This procedure is known as the \emph{tight cut decomposition procedure} and, by a result due to Lov{\'a}sz (see \cite{lovasz1987}), the obtained set of indecomposable graphs is unique, no matter how the tight cuts are chosen in the graph $G$.

When interested in the structure of perfect matchings and the perfect matching polytope, it usually suffices to consider graphs that are connected and where every edge is contained in at least one perfect matching.
These are the \emph{matching covered} graphs.
Many structural properties of matching covered graphs are preserved by tight cut contractions and thus it suffices to consider the bricks and braces in these contexts.
A famous example for this is the problem of Pfaffian graphs (see \cite{mccuaig2004polya} for a comprehensive overview on the topic).

While graphs with perfect matchings are well described by the theorems of Hall (see \cite{lovasz2009matching}) and Tutte (see \cite{tutte1947factorization}), where the latter is a generalisation of the first, the (perfect) matching problem for hypergraphs is one of the the classical \NP-complete problems of Karp (see \cite{karp1972reducibility}).
Still, a better understanding of the (perfect) matching structure of hypergraphs might yield better insight and new techniques for both practical and theoretical applications.

\paragraph*{Contribution}

In this work we study the structure of hypergraphs with perfect matchings.
A hypergraph $H$ is called \emph{matching covered} if it is connected and for every edge $e\in\Fkt{E}{H}$ there is a perfect matching of $H$ containing $e$.
We consider a class of hypergraphs that slightly generalises $r$-uniform hypergraphs: the \emph{uniformable} hypergraphs.
Loosely speaking, a hypergraph is \uniformable if it can be made $r$-uniform for some $r$ by using Berge's concept of vertex multiplication (a formal definition is given in \cref{subsec:preliminaries}).
\Uniformable hypergraphs naturally occur in the process of tight cut contractions where the shore of a tight cut is replaced by a single vertex.
While this operation does not preserve uniformity, if $H$ has a non-trivial tight cut and is \uniformable, then so are its two tight cut contractions (see \cref{cor:preserveuniformable}).

The main result of this work is the uniqueness of the tight cut decomposition for matching covered, \uniformable hypergraphs, which is a direct generalisation of Lov{\'a}sz's famous theorem.
A matching covered, \uniformable hypergraph is called a \emph{hyperbrick}, if it does not have a non-trivial tight cut.

\setcounter{section}{2}
\setcounter{theorem}{8}
\begin{theorem}
	Any two tight cut decomposition procedures of a matching covered, \uniformable hypergraph yield the same list of hyperbricks up to parallel edges.
\end{theorem}
We  provide an example that shows the tight cut decomposition of non-\uniformable matching covered hypergraphs is in general not unique.

In \cref{sec:polytope} we show that the tight cut decomposition of a matching covered hypergraph also induces a decomposition of its matching polytope. Furthermore, we show that balanced \uniformable hypergraphs are closed under our hypergraphic version of tight cut contractions.
\setcounter{section}{3}
\setcounter{theorem}{7}
\begin{theorem}
	Let $H$ be a balanced \uniformable hypergraph and $S\subseteq \Fkt{V}{H}$ such that $\Cut{H}{S}$ is a non-trivial tight cut.
	Then the two contractions $H_S$ and $H_{\Compl{S}}$ are balanced.
\end{theorem}
\setcounter{section}{1}
\setcounter{theorem}{1}

The number of hyperbricks of a matching covered, \uniformable hypergraph is linearly bounded by the number of its vertices.
Moreover, every tight cut decomposition corresponds to a laminar family of tight cuts (see \cref{cor:decompandlaminarfamily}).
Picking a single tight cut from such a family yields two tight cut contractions and these, in turn, contain the remaining tight cuts of the family (see \cref{lemma:tightcut}).
Hence, this result implies that the two problems of finding a non-trivial tight cut in a matching covered, \uniformable hypergraph and finding its tight cut decomposition are polynomial time equivalent. In \cref{sec:algo} we give a polynomial time algorithm finds a non-trivial tight cut in a uniformable and balanced hypergraph if any exists.

\subsection{Preliminaries}\label{subsec:preliminaries}


We start by introducing some definitions and notation.
For a more thorough introduction to the topic we recommend the book on hypergraphs by Berge \cite{bergebook}.

A \emph{hypergraph} $H$ consists of a finite set of vertices and a family of non-empty subsets of vertices 
called edges.
We denote the set of vertices of a hypergraph by $\Fkt{V}{H}$ and its family of edges by 
$\Fkt{E}{H}$.
We call two edges \emph{parallel} if they contain the same set of vertices.

For $r\in\N$ we call a hypergraph \emph{$r$-uniform} if $\Abs{e}=r$ for all edges 
$e\in\Fkt{E}{H}$.

The degree of a vertex $v\in\Fkt{V}{H}$ is the number $\Fkt{d}{v}$ of edges $e\in\Fkt{E}{H}$ with $v\in e$.

Given a subset $S$ of the vertex set there are two possibilities to restrict the edges of $H$ using 
$S$.
First, we can restrict $\Fkt{E}{H}$ by just considering edges lying completely in $S$.
We call the  resulting  hypergraph the \emph{subhypergraph induced by $S$} and denote it by $\PartHG{H}{S}$.
This hypergraph has vertex set $S$ and contains all edges $e\in \Fkt{E}{H}$ with $e\subseteq S$.
If $H'\subseteq H$ is induced by some set $S\subseteq\Fkt{V}{H}$, $H'$ is called an \emph{induced subhypergraph}.

Another possibility is to restrict every edge to $S$ and take the non-empty elements as the new edges.
We denote the hypergraph on $S$ with edges $e\cap S$ for all $e\in \Fkt{E}{H}$ with $e\cap S\neq \emptyset$ by $\SubHG{H}{S}$ and call it the \emph{subhypergraph of $H$ restricted to $S$}.
If $H'\subseteq H$ is restricted to some set $S\subseteq\Fkt{V}{H}$, $H'$ is called a \emph{restricted subhypergraph}.

Similarly, we can define a restriction of the vertex and edge set of $H$ by a set of edges $F\subseteq\Fkt{E}{H}$.
The hypergraph with vertex set $\bigcup F$ and edge set $F$ is called the \emph{partial hypergraph induced by} $F$.
If $H'\subseteq H$ is induced by some set $F\subseteq\Fkt{E}{H}$ of edges, it is called a \emph{partial hypergraph}.

\begin{definition}[Matchings/Perfect Matchings]
	A \emph{matching} in a hypergraph $H$ is a set $M\subseteq\Fkt{E}{H}$ of pairwise disjoint edges and we denote the maximum size of a matching in $H$ by $\MatNum{H}$.
	
	Let $M\subseteq\Fkt{E}{H}$ be a matching of $H$.
	A vertex $x\in\Fkt{V}{H}$ is \emph{covered} by $M$, if there is some edge $e\in M$ that contains $x$.
	Similarly, $x$ is called \emph{exposed} by $M$, if $x$ is not covered by $M$.
	
	A \emph{vertex cover} in $H$ is a set of vertices $T\subseteq\Fkt{E}{H}$ such that every edge of $H$ contains at least one vertex of $T$ and we denote the minimum size of a vertex cover in $H$ by $\VCNum{H}$.
	
	An \emph{edge cover} in $H$ is a set of edges $R\subseteq\Fkt{E}{H}$ such that every vertex of $H$ is contained in an edge of $R$ and we denote the minimum size of an edge cover in $H$ by $\ECNum{H}$.
	
	A maximum matching that is also an edge cover is called \emph{perfect}.
\end{definition}

For $r$-uniform hypergraphs, many edge related parameters coincide.
All maximum cardinality matchings cover the same number of vertices, every perfect matching is also of maximum cardinality, and every perfect matching is a minimum edge cover and vice versa.
While this is not true for all hypergraphs, there is a class that generalises uniform hypergraphs and still has the above mentioned properties.
To describe this class we use the concept of \emph{vertex multiplication} given by Berge \cite{bergebook}.

\begin{definition}
	\emph{Multiplying a vertex $v$} in a hypergraph~$H$ by some natural number~${k\in \N}$ results in the hypergraph with vertex set $\Brace{\Fkt{V}{H} \setminus \Set{v}} \cup \Set{v_1,\ldots, v_k}$ and edge set $$\CondSet{e\in E}{v\notin e}\ \cup \CondSet{e\setminus \Set{v}\cup \Set{v_1,\ldots, v_k}}{e\in E,~v\in e}.$$
	
	Given a function $m : \Fkt{V}{H} \to \Z_{\geq 1}$, the hypergraph obtained from multiplying each vertex $v$ by $\Fkt{m}{v}$ is denoted by $\Multi{H}{m}$.
	We call $\Multi{H}{m}$ a \emph{multiplication} of $H$.
	For $S \subseteq \V{H}$ we introduce the shorthand $\Fkt{m}{S} \coloneqq \sum_{v \in S} \Fkt{m}{v}$.
\end{definition}

Loosely speaking, multiplying a vertex $v$ by $k$ means replacing $v$ by $k$ indistinguishable copies.
Berge also considered the multiplication of a vertex by zero which means that we remove this vertex from all edges containing it.
However, we restrict ourselves to multiplications by positive integers because there is a one-to-one correspondence between the set of matchings in a hypergraph~$H$ and its multiplication~$\Multi{H}{m}$ for functions~$m \geq 1$. Similar, the set of perfect matching in $H$ corresponds one-to-one to the set of perfect matchings in $\Multi{H}{m}$ for $m\geq 1$.

Now, we call a hypergraph~$H$ \emph{\uniformable} if it has a uniform multiplication, i.e.\@ there exists a function $m : \Fkt{V}{H}\rightarrow \Z_{\geq 1}$ and an integer $r \in \Z$ such that $\sum_{v \in e}\Fkt{m}{v} = r$ for all $e\in \Fkt{E}{H}$.

We can decide in polynomial time whether a hypergraph $H$ is \uniformable or not. Namely, $H$ is \uniformable if and only if the system
\begin{align}
	\sum_{v\in e}m(v)=r &\ \forall e\in E(H) \label{unieq1}\\
	m(v)\geq 1 &\ \forall v\in V(H) \label{unieq2}
\end{align}
has an integral solution. This is the case if and only if it has a fractional solution $m\in \mathds Q^{V(H)}$, $r\in \mathds Q$ because $d\cdot m$, $d\cdot r$ is an integral solution to \cref{unieq1}-\cref{unieq2} if $d\in \mathds Z$ is chosen such that $d\cdot m(v)$ is integral for all $v\in V(H)$. Deciding whether a system of linear equations and inequalities has a fractional solution can be done in polynomial time.

The class of \uniformable hypergraphs is a natural extension of uniform hypergraphs in the matching setting.
First, this is because of the one-to-one correspondence of perfect matchings as mentioned above and, second, the tight cut contraction which is defined in the next subsection, shrinks the shore of a tight cut to a single vertex and thus does not preserve uniformity.
However, as \cref{cor:preserveuniformable} shows, uniformability is preserved.

We are especially interested in a more restricted class of \uniformable hypergraphs: those where every edge lies in a perfect matching.

\begin{definition}
	A hypergraph $H$ is called \emph{matching covered}, if it is connected, $\Abs{\Fkt{E}{H}}\geq 1$, and every edge of $H$ is contained in a perfect matching.
\end{definition}

The hypergraphs we consider in this paper will always be connected and have at least one perfect matching. As edges not contained in any perfect matching do not play a role in the structure of perfect matchings of a hypergraph, we assume throughout the paper that all hypergraphs are matching covered.

\subsection{Tight Cuts and their Contractions}

Our goal is the generalisation of tight cuts to the hypergraphic setting.
This subsection is dedicated to the new definitions and includes some basic results on tight cut contractions including the preservation of uniformability.

\begin{definition}
	Let $H$ be a matching covered, \uniformable hypergraph, and $S\subseteq \Fkt{V}{H}$.
	The cut $\Cut{H}{S}$ is tight if $\Abs{M\cap \Cut{H}{S}}=1$ for all perfect matchings $M$ of $H$.
	A cut $\Cut{H}{S}$ is \emph{trivial} if $\Abs{S}=1$ or $\Abs{\Compl{S}} = 1$, otherwise it is called \emph{non-trivial}.
	Here, $\Compl{S}$ denotes the complement set to $S$, namely $\V{H}\setminus S$.
\end{definition}

We need the following simple observations concerning cuts and tight cuts in \uniformable hypergraphs.
Both are straightforward generalisations of the graph case.

\begin{proposition}\label{obs:tightcut}
	Let $H$ be a matching covered, \uniformable hypergraph with a function $m:\Fkt{V}{H}\to \Z_{\geq 1}$ such that $\Multi{H}{m}$ is $r$-uniform for some $r\in \Z$.
	\begin{enumerate}
		\item If $\Cut{H}{S}$ is a tight cut, then $\sum_{v\in e\cap S} \Fkt{m}{v} = k$ for all $e\in \Cut{H}{S}$, where $k$ is the unique integer with $k\in \Set{1,\ldots, r-1}$ and $k\equiv_r \sum_{v\in S}\Fkt{m}{v}$.
		
		\item If $S\subseteq \Fkt{V}{H}$ with $\sum_{v\in S}\Fkt{m}{v} \equiv_r 0$, then $\Abs{M\cap \Cut{H}{S}} \geq 2$ or $\Abs{M\cap \Cut{H}{S}} = 0$ for all perfect matchings $M$ of $H$.
			If additionally $S\notin \Set{\emptyset, \Fkt{V}{H}}$, then there exists a perfect matching $M$ with $\Abs{M\cap \Cut{H}{S}} \geq 2$.
	\end{enumerate}
\end{proposition}

\begin{proof}
	For the first observation let $\Cut{H}{S}$ be a tight cut, $e'\in \Cut{H}{S}$ arbitrary, and $M$ a perfect matching containing $e'$.
	The set $M_S \coloneqq\CondSet{e\subseteq S}{e\in M}$ is a matching of $\InducedSubgraph{H}{S}$ covering $S\setminus e'$.
	Summing the function $m$ over all $v\in \Fkt{V}{M_S}$ we get $\sum_{v\in S}\Fkt{m}{v}-\sum_{v\in e'\cap S}\Fkt{m}{v}$.
	On the other hand, $\sum_{v\in \Fkt{V}{M_S}}\Fkt{m}{v}=\sum_{e\in M_S}\sum_{v\in e}\Fkt{m}{v}=\Abs{M_S}\cdot r$.
	Together, we get $\sum_{v\in S}\Fkt{m}{v}\equiv_r \sum_{v\in e'\cap S}\Fkt{m}{v}$.
	As $\sum_{v\in e'\cap S}\Fkt{m}{v}\leq \sum_{v\in e'}\Fkt{m}{v}-1$ the claim follows.
	
	For the second part of the observation, let $S\subseteq \Fkt{V}{H}$.
	If $S=\emptyset$ or $S=\Fkt{V}{H}$, then $\Cut{H}{S}=\emptyset$ and thus $\Abs{M\cap \Cut{H}{S}} = 0$ for all perfect matchings $M$ of $H$.
	In the remainder of the proof we assume that $S\notin \Set{\emptyset, \Fkt{V}{H}}$.
	This assumption together with the fact that $H$ is connected implies that $\Cut{H}{S}\neq \emptyset$.
	The hypergraph $H$ is matching covered which implies that for every $e'\in \Cut{H}{S}$ there exists a perfect matching containing $e'$. Suppose that there exists a perfect matching $M$ intersecting $\Cut{H}{S}$ in exactly one edge~$e'$, then $\sum_{v\in e'\cap S}\Fkt{m}{v}\equiv \sum_{v\in S}\Fkt{m}{v}\equiv_r 0$ follows.
	As $\Fkt{m}{v} \geq 1$ for all $v\in \V{H}$ and $\sum_{v\in e'}\Fkt{m}{v}=r$, we get $e'\cap S=\emptyset$, contradicting $e'\in \Cut{H}{S}$.
	Thus, if $M$ is a perfect matching intersecting $\Cut{H}{S}$, then $\Abs{M\cap \Cut{H}{S}}\geq 2$. 
\end{proof}

Next, we define tight cut contractions.
The basic idea is to replace a shore $S$ of a tight cut by a single vertex $s$ which is included in all edges in $\Cut{H}{S}$ replacing the vertices of $S$.
For an illustration see \cref{fig:thm_tight_cut_contr}.
Note that this procedure possibly introduces parallel edges.

\begin{definition}[tight cut contractions]
	Let $H$ be a hypergraph and $\Cut{H}{S}$ a tight cut of $H$.
	For every $e\in \Cut{H}{S}$ we define two new edges $e_s = \Brace{e\setminus S}\cup \Set{s}$ and $e_{\Compl s} = \Brace{e\cap S}\cup \Set{\Compl{s}}$ where $s,\Compl{s}\notin \Fkt{V}{H}$ are two new vertices representing $S$ and $\Compl{S}$, respectively.
	
	The \emph{tight cut contractions} of $H$ w.r.t.\ $S$ and $\Compl{S}$ are the hypergraphs $H_S$ and $H_{\Compl{S}}$ where 
	$\Fkt{V}{H_S}\coloneqq\Set{s}\cup \Compl{S}$ and $\Fkt{V}{H_{\Compl{S}}} \coloneqq \Set{\Compl{s}}\cup S$ are the vertices, and the edges are given by $\Fkt{E}{H_S}\coloneqq\CondSet{e\in \Fkt{E}{H}}{e\subseteq \Compl{S}}\cup \CondSet{e_s}{e\in \Cut{H}{S}}$ and $\Fkt{E}{H_{\Compl{S}}} \coloneqq \CondSet{e\in \Fkt{E}{H}}{e\subseteq S} \cup \CondSet{e_{\Compl{s}}}{e\in\Cut{H}{S}}$.
\end{definition}

\begin{figure}[t]
	\centering
	\begin{tikzpicture}[scale=0.85]
	\node (M) at (0,0) {};
	\node (AR) at ($(M)+(0:4.5)$) {
		\begin{tikzpicture}[scale=0.65]
		\node (B) at (0,0) {};
		\def\OuterCurve{%
			let \p1=($(B)+(90:3)$),\p2=($(B)+(270:3)$) in (-4,\y1) rectangle (4,\y2)
		}
		\coordinate (Upper) at ($(B)+(90:3)$);
		\coordinate (Lower) at ($(B)+(270:3)$);
		\def\CutCurve{%
			(Upper) to[out=-40, in=100] ($(B)+(87:1)$) to[out=280, in=110] (Lower)
		}
		\draw ($(B)+(90:2)$) edge[Hedge,closed,curve through={($(B)+(135:2.5)$) ($(B)+(180:3)$) ($(B)+(225:2.5)$) ($(B)+(270:1.8)$) ($(B)+(315:2.2)$) ($(B)+(360:2.4)$)}] ($(B)+(45:2.7)$);
		\node (e2) at ($(B)+(130:1.2)$) {};
		\draw ($(e2)+(90:0.8)$) edge[Medge,closed,curve through={($(e2)+(135:0.6)$) ($(e2)+(180:0.4)$) ($(e2)+(225:0.55)$) ($(e2)+(270:0.6)$) ($(e2)+(315:0.7)$) ($(e2)+(360:0.6)$) }] ($(e2)+(45:0.7)$);
		\node (e3) at ($(B)+(165:2)$) {};
		\draw ($(e3)+(90:0.8)$) edge[Medge,closed,curve through={($(e3)+(135:0.6)$) ($(e3)+(180:0.8)$) ($(e3)+(225:0.8)$) ($(e3)+(270:0.4)$) ($(e3)+(315:0.8)$) ($(e3)+(360:0.6)$) }] ($(e3)+(45:0.7)$);
		\node (e4) at ($(B)+(52:1.9)$) {};
		\draw ($(e4)+(90:0.45)$) edge[Medge,closed,curve through={($(e4)+(135:0.6)$) ($(e4)+(180:0.8)$) ($(e4)+(225:0.3)$) ($(e4)+(270:0.5)$) ($(e4)+(315:0.4)$) ($(e4)+(360:0.7)$) }] ($(e4)+(45:0.5)$);
		\node (e5) at ($(B)+(15:1.7)$) {};
		\draw ($(e5)+(90:0.5)$) edge[Medge,closed,curve through={($(e5)+(135:0.4)$) ($(e5)+(180:0.55)$) ($(e5)+(225:0.4)$) ($(e5)+(270:0.5)$) ($(e5)+(315:0.8)$) ($(e5)+(360:0.5)$) }] ($(e5)+(45:0.7)$);
		\node (e6) at ($(B)+(320:1.5)$) {};
		\draw ($(e6)+(90:0.8)$) edge[Medge,closed,curve through={($(e6)+(135:0.3)$) ($(e6)+(180:1)$) ($(e6)+(225:0.8)$) ($(e6)+(270:0.5)$) ($(e6)+(315:0.4)$) ($(e6)+(360:0.5)$) }] ($(e6)+(45:0.7)$);
		\node (e7) at ($(B)+(205:2.1)$) {};
		\draw ($(e7)+(90:0.7)$) edge[Medge,closed,curve through={($(e7)+(135:0.7)$) ($(e7)+(180:0.5)$) ($(e7)+(225:0.5)$) ($(e7)+(270:0.6)$) ($(e7)+(315:1)$) ($(e7)+(360:0.4)$) }] ($(e7)+(45:0.3)$);
		\begin{scope}
		\clip[use Hobby shortcut, closed=true] \OuterCurve;
		\clip \CutCurve
		-- (current bounding box.south |- Lower)
		-- (current bounding box.south west)
		-- (current bounding box.north west)
		-- (current bounding box.north |- Upper)
		-- cycle
		;
		\fill[white] (current bounding box.south west)
		rectangle (current bounding box.north east)
		;
		\end{scope}
		\node[label distance=2mm,label=3:{$e_{s}$}] (e1) at ($(B)+(225:0.3)$) {};
		\draw ($(e1)+(90:0.2)$) edge[Medge,closed,curve through={($(e1)+(135:0.4)$) ($(e1)+(180:1)$) ($(e1)+(225:1.1)$) ($(e1)+(270:0.4)$) ($(e1)+(315:0.7)$) ($(e1)+(360:1)$) ($(e1)+(45:1.5)$)}] ($(e1)+(60:1.6)$);
		\node[scale=0.9] (a1) at ($(B)+(-1.1,0.8)$) {};
		
		\node[scale=0.9,vertex] (a2) at ($(B)+(-0.8,-0.4)$) {};
		\node[opacity=0.8,fill=white,shape=ellipse] at ($(a2)-(0.6,+0.1)$) {$s$};
		\node[scale=0.9,vertex] (a2) at ($(B)+(-0.8,-0.4)$) {};
		
		
		\draw[BostonUniversityRed,thick,use Hobby shortcut, closed=false] \CutCurve;
		\node[CornflowerBlue] (H2) at ($(B)+(335:3)$) {$H_S$};
		\node[DarkTangerine] (M2) at ($(B)+(300:2.8)$) {$M_S$};
		\end{tikzpicture}};
	\node (AL) at ($(M)+(180:4.5)$) {
		\begin{tikzpicture}[scale=0.65]
		\node (B) at (0,0) {};
		\def\OuterCurve{%
			($(B)+(90:3)$) .. ($(B)+(180:3.5)$) ..
			($(B)+(270:3)$) .. ($(B)+(360:3.5)$) .. ($(B)+(90:3)$)
		}
		\coordinate (Upper) at ($(B)+(90:3)$);
		\coordinate (Lower) at ($(B)+(270:3)$);
		\def\CutCurve{%
			(Upper) to[out=-40, in=100] ($(B)+(87:1)$) to[out=280, in=110] (Lower)
		}
		\draw ($(B)+(90:2)$) edge[Hedge,closed,curve through={($(B)+(135:2.5)$) ($(B)+(180:3)$) ($(B)+(225:2.5)$) ($(B)+(270:1.8)$) ($(B)+(315:2.2)$) ($(B)+(360:2.4)$)}] ($(B)+(45:2.7)$);
		\node (e2) at ($(B)+(130:1.2)$) {};
		\draw ($(e2)+(90:0.8)$) edge[Medge,closed,curve through={($(e2)+(135:0.6)$) ($(e2)+(180:0.4)$) ($(e2)+(225:0.55)$) ($(e2)+(270:0.6)$) ($(e2)+(315:0.7)$) ($(e2)+(360:0.6)$) }] ($(e2)+(45:0.7)$);
		\node (e3) at ($(B)+(165:2)$) {};
		\draw ($(e3)+(90:0.8)$) edge[Medge,closed,curve through={($(e3)+(135:0.6)$) ($(e3)+(180:0.8)$) ($(e3)+(225:0.8)$) ($(e3)+(270:0.4)$) ($(e3)+(315:0.8)$) ($(e3)+(360:0.6)$) }] ($(e3)+(45:0.7)$);
		\node (e4) at ($(B)+(52:1.9)$) {};
		\draw ($(e4)+(90:0.45)$) edge[Medge,closed,curve through={($(e4)+(135:0.6)$) ($(e4)+(180:0.8)$) ($(e4)+(225:0.3)$) ($(e4)+(270:0.5)$) ($(e4)+(315:0.4)$) ($(e4)+(360:0.7)$) }] ($(e4)+(45:0.5)$);
		\node (e5) at ($(B)+(15:1.7)$) {};
		\draw ($(e5)+(90:0.5)$) edge[Medge,closed,curve through={($(e5)+(135:0.4)$) ($(e5)+(180:0.55)$) ($(e5)+(225:0.4)$) ($(e5)+(270:0.5)$) ($(e5)+(315:0.8)$) ($(e5)+(360:0.5)$) }] ($(e5)+(45:0.7)$);
		\node (e6) at ($(B)+(320:1.5)$) {};
		\draw ($(e6)+(90:0.8)$) edge[Medge,closed,curve through={($(e6)+(135:0.3)$) ($(e6)+(180:1)$) ($(e6)+(225:0.8)$) ($(e6)+(270:0.5)$) ($(e6)+(315:0.4)$) ($(e6)+(360:0.5)$) }] ($(e6)+(45:0.7)$);
		\node (e7) at ($(B)+(205:2.1)$) {};
		\draw ($(e7)+(90:0.7)$) edge[Medge,closed,curve through={($(e7)+(135:0.7)$) ($(e7)+(180:0.5)$) ($(e7)+(225:0.5)$) ($(e7)+(270:0.6)$) ($(e7)+(315:1)$) ($(e7)+(360:0.4)$) }] ($(e7)+(45:0.3)$);
		\begin{scope}
		\clip[use Hobby shortcut, closed=true] \OuterCurve;
		\clip \CutCurve
		-- (current bounding box.south |- Lower)
		-- (current bounding box.south east)
		-- (current bounding box.north east)
		-- (current bounding box.north |- Upper)
		-- cycle
		;
		\fill[white] (current bounding box.south west)
		rectangle (current bounding box.north east)
		;
		\end{scope}
		\node[label distance=2mm,label=183:{$e_{\Compl{s}}$}] (e1) at ($(B)+(225:0.3)$) {};
		\draw ($(e1)+(90:0.2)$) edge[Medge,closed,curve through={($(e1)+(135:0.4)$) ($(e1)+(180:1)$) ($(e1)+(225:1.1)$) ($(e1)+(270:0.4)$) ($(e1)+(315:0.7)$) ($(e1)+(360:1)$) ($(e1)+(45:1.5)$)}] ($(e1)+(60:1.6)$);
		\node (e2) at ($(B)+(130:1.2)$) {};
		\node[scale=0.9] (ak1) at ($(B)+(0.6,-0.3)$) {};
		\node[opacity=0.8,fill=white,shape=ellipse] at ($(ak1)+(0.4,0)$) {$\Compl{s}$};
		\node[scale=0.9,vertex] at ($(B)+(0.4,-0.3)$) {};
		
		\draw[BostonUniversityRed,thick,use Hobby shortcut, closed=false] \CutCurve;
		\node[CornflowerBlue] (H1) at ($(B)+(290:2.5)$) {$H_{\Compl{S}}$};
		\node[DarkTangerine] (M1) at ($(B)+(240:2.7)$) {$M_{\Compl{S}}$};
		\end{tikzpicture}};
	\node (A) at (M){
		\begin{tikzpicture}[scale=0.65]
		\node (B) at (0,0) {};
		\def\OuterCurve{%
			($(B)+(90:3)$) .. ($(B)+(180:3.5)$) ..
			($(B)+(270:3)$) .. ($(B)+(360:3.5)$) .. ($(B)+(90:3)$)
		}
		\draw ($(B)+(90:2)$) edge[Hedge,closed,curve through={($(B)+(135:2.5)$) ($(B)+(180:3)$) ($(B)+(225:2.5)$) ($(B)+(270:1.8)$) ($(B)+(315:2.2)$) ($(B)+(360:2.4)$)}] ($(B)+(45:2.7)$);
		\coordinate (Upper) at ($(B)+(90:3)$);
		\coordinate (Lower) at ($(B)+(270:3)$);
		\def\CutCurve{%
			(Upper) to[out=-40, in=100] ($(B)+(87:1)$) to[out=280, in=110] (Lower)
		}
		\node[label=west:{$e$}] (e1) at ($(B)+(225:0.3)$) {};
		\draw ($(e1)+(90:0.2)$) edge[Medge,closed,curve through={($(e1)+(135:0.4)$) ($(e1)+(180:1)$) ($(e1)+(225:1.1)$) ($(e1)+(270:0.4)$) ($(e1)+(315:0.7)$) ($(e1)+(360:1)$) ($(e1)+(45:1.5)$)}] ($(e1)+(60:1.6)$);
		\node (e2) at ($(B)+(130:1.2)$) {};
		\draw ($(e2)+(90:0.8)$) edge[Medge,closed,curve through={($(e2)+(135:0.6)$) ($(e2)+(180:0.4)$) ($(e2)+(225:0.55)$) ($(e2)+(270:0.6)$) ($(e2)+(315:0.7)$) ($(e2)+(360:0.6)$) }] ($(e2)+(45:0.7)$);
		\node (e3) at ($(B)+(165:2)$) {};
		\draw ($(e3)+(90:0.8)$) edge[Medge,closed,curve through={($(e3)+(135:0.6)$) ($(e3)+(180:0.8)$) ($(e3)+(225:0.8)$) ($(e3)+(270:0.4)$) ($(e3)+(315:0.8)$) ($(e3)+(360:0.6)$) }] ($(e3)+(45:0.7)$);
		\node (e4) at ($(B)+(52:1.9)$) {};
		\draw ($(e4)+(90:0.45)$) edge[Medge,closed,curve through={($(e4)+(135:0.6)$) ($(e4)+(180:0.8)$) ($(e4)+(225:0.3)$) ($(e4)+(270:0.5)$) ($(e4)+(315:0.4)$) ($(e4)+(360:0.7)$) }] ($(e4)+(45:0.5)$);
		\node (e5) at ($(B)+(15:1.7)$) {};
		\draw ($(e5)+(90:0.5)$) edge[Medge,closed,curve through={($(e5)+(135:0.4)$) ($(e5)+(180:0.55)$) ($(e5)+(225:0.4)$) ($(e5)+(270:0.5)$) ($(e5)+(315:0.8)$) ($(e5)+(360:0.5)$) }] ($(e5)+(45:0.7)$);
		\node (e6) at ($(B)+(320:1.5)$) {};
		\draw ($(e6)+(90:0.8)$) edge[Medge,closed,curve through={($(e6)+(135:0.3)$) ($(e6)+(180:1)$) ($(e6)+(225:0.8)$) ($(e6)+(270:0.5)$) ($(e6)+(315:0.4)$) ($(e6)+(360:0.5)$) }] ($(e6)+(45:0.7)$);
		\node (e7) at ($(B)+(205:2.1)$) {};
		\draw ($(e7)+(90:0.7)$) edge[Medge,closed,curve through={($(e7)+(135:0.7)$) ($(e7)+(180:0.5)$) ($(e7)+(225:0.5)$) ($(e7)+(270:0.6)$) ($(e7)+(315:1)$) ($(e7)+(360:0.4)$) }] ($(e7)+(45:0.3)$);

		\draw[BostonUniversityRed,thick,use Hobby shortcut, closed=false] \CutCurve;
		\node[CornflowerBlue] (H) at ($(B)+(45:3.1)$) {$H$};
		\node[DarkTangerine] (M) at ($(B)+(360:2.9)$) {$M$};
		\node[BostonUniversityRed!80!black,text height=2ex] (S) at ($(B)+(100:2.9)$) {$S$};
		\node[BostonUniversityRed!80!black,text height=2ex] (SC) at ($(S)+(360:1.3)$) {$\Compl{S}$};
		\end{tikzpicture}};
	\node (pA) at (A) {};
	\node (pAL) at (AL) {};
	\node (pAR) at (AR) {};
	\draw[shorten <=2cm,shorten >=1.1cm,->,-latex,thick,dashed] (pA) -- (pAL);
	\draw[shorten <=1.9cm,shorten >=1.2cm,->,-latex,thick,dashed] (pA) -- (pAR);
\end{tikzpicture}
\caption{A \textcolor{DarkTangerine}{perfect matching} in a matching covered, \uniformable hypergraph $H$ and the tight cut contractions $H_S$ and $H_{\Compl{S}}$.}
\label{fig:thm_tight_cut_contr}
\end{figure}
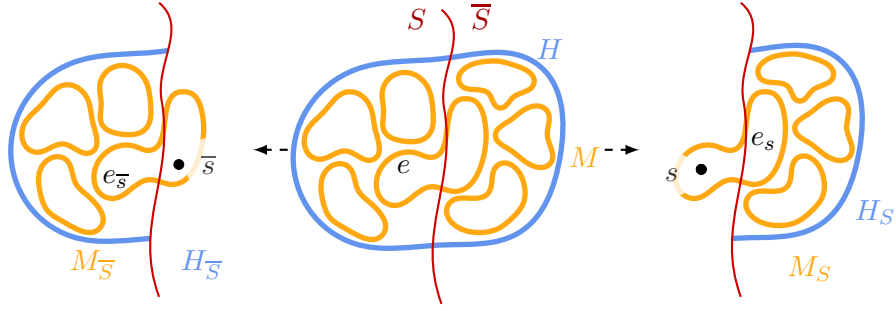

Tight cut contractions are useful because there is a correspondence between perfect matchings of a hypergraph $H$ and pairs of perfect matchings in the tight cut contractions $H_S$ and $H_{\Compl{S}}$ implying that the tight cut contractions of a matching covered hypergraph remain matching covered.

\begin{proposition}\label{prop:matcoveredtight}
	Let $H$ be a hypergraph with a perfect matching and $S \subseteq \Fkt{V}{H}$ be a set of vertices 
	defining a tight cut $\Cut{H}{S}$.
	Then $H$ is matching covered if and only if $H_S$ and $H_{\Compl{S}}$ are matching covered.
\end{proposition}

\begin{proof}
	We say that a perfect matching $M_S$ of $H_S$ agrees with a perfect matching $M_{\Compl{S}}$ of 
	$H_{\Compl{S}}$ on 
	$\Cut{H}{S}$ if there exists an edge $e\in \Cut{H}{S}$ such that $e_s$ is the unique edge in $M_S\cap \Cut{H_S}{s}$ and $e_{\Compl{s}}$ is the unique edge in $M_{\Compl{S}}\cap \Cut{H_{\Compl{S}}}{\Compl{s}}$.
	In this case, $M\coloneqq \Brace{M_S\cup M_{\Compl{S}}\cup \Set{e}}\setminus \Set{e_s, e_{\Compl{s}}}$ is a matching of $H$ covering $\Fkt{V}{M_S \setminus \Set{e_s}} \cup \Fkt{V}{M_{\Compl{S}} \setminus \Set{e_{\Compl{s}}}} \cup e = \Brace{\Compl{S} \setminus e} \cup \Brace{S \setminus e} \cup e = \Fkt{V}{H}$.

	Reversely, let $M$ be a perfect matching in $H$ and $e$ the unique edge in $M \cap \Cut{H}{S}$.
	Consider $M_S \coloneqq \Set{e' \in M \mid e' \subseteq \Compl{S}} \cup \Set{e_s}$, and $M_{\Compl{S}} \coloneqq \Set{e' \in M \mid e' \subseteq S} \cup \Set{e_{\Compl{s}}}$.
	It is straightforward to verify that $M_S$ and $M_{\Compl{S}}$ are perfect matchings of $H_S$ and $H_{\Compl{S}}$ agreeing on $\Cut{H}{S}$.
	
	The correspondence between perfect matchings in $H$ and pairs of perfect matchings in $H_S$, $H_{\Compl{S}}$ agreeing on $\Cut{H}{S}$ implies that $H$ is matching covered if and only if $H_S$ and $H_{\Compl{S}}$ are matching covered.
\end{proof}

Performing a tight cut contraction on a uniform hypergraph does not yield uniform tight cut contractions in general.
However, \cref{obs:tightcut} implies that the tight cut contractions of a \uniformable hypergraph are again \uniformable.

\begin{corollary}\label{cor:preserveuniformable}
	Let $H$ be a matching covered, \uniformable hypergraph, and $\Cut{H}{S}$ a tight cut.
	The two contractions $H_S$ and $H_{\Compl S}$ are matching covered and \uniformable.
\end{corollary}

\begin{proof}
	By \cref{prop:matcoveredtight} the hypergraphs $H_S$ and $H_{\Compl S}$ are matching covered.
	It remains to show that they are \uniformable.
	Let $m : \Fkt{V}{H} \to \Z_{\geq 1}$ such that $\Multi{H}{m}$ is $r$-uniform for some $r\in \Z$.
	By \cref{obs:tightcut} there exists an integer $k\in \Set{1,\ldots, r-1}$ such that $\sum_{v\in S}\Fkt{m}{v}\equiv_r k$ and $\sum_{v \in e\cap S}\Fkt{m}{v}=k$ for all $e\in \Cut{H}{S}$.
	Using $m$ we can define a function $m_S : \Fkt{V}{H_S} \to \Z_{\geq 1}$ such that $\Multi{H_S}{m_S}$ is $r$-uniform.
	Namely, $\Fkt{m_S}{v} \coloneqq \Fkt{m}{v}$ for all $v\in \Compl{S}$ and $\Fkt{m_S}{s} \coloneqq k$.
	If $e\in \E{H_S}$ with $e\subseteq \Compl{S}$, then $\sum_{v\in e} \Fkt{m_S}{v}=\sum_{v\in e}\Fkt{m}{v}=r$.
	All other edges of $H_S$ are of the form $e\setminus S\cup \Set{s}$ for some $e\in \Cut{H}{S}$.
	For those edges we have 
	\begin{align*}
		\sum_{v\in e\setminus S} \Fkt{m_S}{v} + \Fkt{m_S}{s} = \sum_{v\in e\setminus S}\Fkt{m}{v}+k=\sum_{v\in e}\Fkt{m}{v}-\sum_{v\in e\cap S}\Fkt{m}{v}+k=r-k+k=r.
	\end{align*}
	Thus, $\Multi{H_S}{m_S}$ is $r$-uniform.
	Similarly, if we define $m_{\Compl{S}} : \Fkt{V}{H_{\Compl{S}}}\to \Z_{\geq 1}$ by $\Fkt{m_{\Compl{S}}}{v} \coloneqq\Fkt{m}{v}$ for all $v\in S$ and $\Fkt{m_{\Compl{S}}}{\Compl{s}}\coloneqq r-k$, then $\Multi{H_{\Compl{S}}}{m_{\Compl{S}}}$ is $r$-uniform. 
\end{proof}

\section{The Tight Cut Decomposition}
\label{sec:tightCutDecomp}

As we have seen, tight cuts and the tight cut contraction can be generalised to hypergraphs.
However, the properties of tight cut contractions and the tight cut decomposition do not carry over to the world of hypergraphs as easily.

In this section we formally introduce the tight cut decomposition of hypergraphs.
Then we will point out some of the difficulties in the properties of tight cuts and their contractions and give an example of a matching covered hypergraph that has two distinct tight cut decompositions.
As the uniqueness for general hypergraphs cannot exist, we then focus on the \uniformable case in order to prove our main theorem.

First, we define formally what we mean by a tight cut decomposition by considering \cref{alg:tcd}.
It takes a matching covered hypergraph as an input and decomposes it along a non-trivial tight cut if any exists.
In this way, two new matching covered hypergraphs arise.
If at least one of them has a non-trivial tight cut, then the algorithm contracts again both shores of this cut.
In each iteration of the while loop the number of obtained hypergraphs increases by one.
As we always choose a non-trivial tight cut in \cref{chooseTightcut}, the hypergraphs $\Brace{H_j}_S$ and $\Brace{H_j}_{\Compl{S}}$ have less vertices than $H_j$.
This implies that the algorithm will eventually terminate because a hypergraph with less than four vertices has no non-trivial tight cuts.

\begin{algorithm}
  \caption{Tight Cut Decomposition Procedure}
  \label{alg:tcd}
  \begin{algorithmic}[1]
  \Procedure{Tight Cut Decomposition}{$H$}
  \State $i\leftarrow 1$
  \State $H_1\leftarrow H$
  \While {$\exists j\in\Set{1,\ldots, i}$ s.t.\ $H_j$ has a non-trivial tight cut.}
  \State Choose any such $H_j$ and $S\subseteq \V{H_j}$ defining a non-trivial tight cut, \label{chooseTightcut}
  \State $H_j\leftarrow \Brace{H_j}_{\Compl{S}}$,
  \State $H_{i+1}\leftarrow \Brace{H_j}_{S}$,
  \State $i\leftarrow i+1$.
  \EndWhile
  \State \textbf{return} $H_1,\ldots, H_i$.
  \EndProcedure
  \end{algorithmic}
  \end{algorithm}
   
  Now, a \emph{tight cut decomposition} of a matching covered hypergraph $H$ is the output obtained from \cref{alg:tcd}.
  This means that a tight cut decomposition consists of a list of matching covered hypergraphs without non-trivial tight cuts that where obtained from $H$ by successive tight cut contractions.
  \Cref{fig:tcd} shows an example of the tight cut decomposition of a matching covered $3$-uniform hypergraph $H$.
  
\begin{figure}[h!]
  	\centering
  	\begin{tikzpicture}[scale=0.9]
	  	\def\Hdist{1.3}
	  	\def\Vdist{1.2}
	  	\path (-\Hdist/2,\Vdist/2) rectangle (5.5*\Hdist,-2.5*\Vdist);
	  	\definecolor{blue3}{rgb}{0.54, 0.81, 0.94}
	  	\definecolor{blue2}{rgb}{0.0, 0.5, 1.0} 
	  	\definecolor{blue4}{rgb}{0.0, 0.5, 0.69}
	  		  	
	  	\definecolor{CutRed}{rgb}{0.76, 0.23, 0.13}
	  	\definecolor{CutViolet}{rgb}{0.55, 0.0, 0.55}
	  	\definecolor{CutLightGreen}{rgb}{0.13, 0.55, 0.13}
	  	\definecolor{CutDarkGreen}{rgb}{0.0, 0.42, 0.24}
	  	\definecolor{CutPink}{rgb}{1.0, 0.11, 0.81}
	  	\colorlet{CutRedText}{CutRed!80!black}
	  	\colorlet{CutVioletText}{CutViolet!80!black}
	  	\colorlet{CutLightGreenText}{AO}
	  	\colorlet{CutDarkGreenText}{CutDarkGreen!80!black}
	  	\colorlet{CutPinkText}{CutPink!80!black}
	  	
	  	\definecolor{inchworm}{rgb}{0.7, 0.93, 0.36}
	  	
	  	\tikzset{Brace/.style={shape=circle, fill=inchworm!70!white,inner sep=0.4mm}}
	  	
	  	\node[vertex] (A1) at (0,0) {};
	  	\node at ($(A1)+(-{\Hdist*(2/3)},{\Vdist*(1/3)})$) {\textcolor{blue4}{$H$}};
	  	\node[vertex] (A2) at ($(A1)+(\Hdist,0)$) {};
	  	\node[vertex] (A3) at ($(A1)+(0,-\Vdist)$) {};
	  	\node[vertex] (A4) at ($(A3)+(\Hdist,0)$) {};
	  	\node[vertex] (A5) at ($(A3)+(0,-\Vdist)$) {};
	  	\node[vertex] (A6) at ($(A5)+(\Hdist,0)$) {};
	  	
	  	\node[vertex] (B1) at ($(A2)+(\Hdist,0)$) {};
	  	\node[vertex] (B2) at ($(B1)+(\Hdist,0)$) {};
	  	\node[vertex] (B3) at ($(B1)+(0,-\Vdist)$) {};
	  	\node[vertex] (B4) at ($(B3)+(\Hdist,0)$) {};
	  	\node[vertex] (B5) at ($(B3)+(0,-\Vdist)$) {};
	  	\node[vertex] (B6) at ($(B5)+(\Hdist,0)$) {};
	  	
	  	\node[vertex] (C1) at ($(B2)+(\Hdist,0)$) {};
	  	\node[vertex] (D1) at ($(C1)+(\Hdist,0)$) {};
	  	\node[vertex] (C2) at ($(C1)+(0,-\Vdist)$) {};
	  	\node[vertex] (D2) at ($(C2)+(\Hdist,0)$) {};
	  	\node[vertex] (C3) at ($(C2)+(0,-\Vdist)$) {};
	  	\node[vertex] (D3) at ($(C3)+(\Hdist,0)$) {};
	  	
	  	\draw ($(A1)+(90:0.3)$) edge[Hedge,closed,curve through={($(A1)+(30:0.2)$) ($(A3)+(360:0.25)$) ($(A5)+(320:0.2)$) ($(A5)+(270:0.3)$) ($(A5)+(210:0.2)$) ($(A3)+(180:0.25)$)}] ($(A1)+(150:0.2)$);
	  	\draw ($(A2)+(90:0.3)$) edge[Hedge,blue3,closed,curve through={($(A2)+(30:0.2)$) ($(A4)+(360:0.25)$) ($(A6)+(320:0.2)$) ($(A6)+(270:0.3)$) ($(A6)+(210:0.2)$) ($(A4)+(180:0.25)$)}] ($(A2)+(150:0.2)$);
	  	\draw ($(C1)+(90:0.3)$) edge[Hedge,closed,curve through={($(C1)+(30:0.2)$) ($(C2)+(360:0.25)$) ($(C3)+(320:0.2)$) ($(C3)+(270:0.3)$) ($(C3)+(210:0.2)$) ($(C2)+(180:0.25)$)}] ($(C1)+(150:0.2)$);
	  	\draw ($(A2)+(180:0.3)$) edge[Hedge,closed,curve through={($(A2)+(120:0.2)$) ($(B1)+(90:0.25)$) ($(B2)+(60:0.2)$) ($(B2)+(360:0.3)$) ($(B2)+(300:0.2)$) ($(B1)+(270:0.25)$)}] ($(A2)+(240:0.2)$);
	  	\draw ($(A4)+(180:0.3)$) edge[Hedge,closed,curve through={($(A4)+(120:0.2)$) ($(B3)+(90:0.25)$) ($(B4)+(60:0.2)$) ($(B4)+(360:0.3)$) ($(B4)+(300:0.2)$) ($(B3)+(270:0.25)$)}] ($(A4)+(240:0.2)$);
	  	\draw ($(A6)+(180:0.3)$) edge[Hedge,closed,curve through={($(A6)+(120:0.2)$) ($(B5)+(90:0.25)$) ($(B6)+(60:0.2)$) ($(B6)+(360:0.3)$) ($(B6)+(300:0.2)$) ($(B5)+(270:0.25)$)}] ($(A6)+(240:0.2)$);
	  	
	  	\draw ($(D2)+(45:0.2)$) edge[Hedge,blue3,closed,curve through={($(D2)+(360:0.2)$) ($(D3)+(360:0.2)$) ($(D3)+(315:0.2)$) ($(D3)+(270:0.2)$) ($(D2)+(225:1.2)$) ($(C2)+(180:0.2)$) ($(C2)+(135:0.2)$) ($(C2)+(90:0.2)$) }] ($(D2)+(90:0.2)$);
	  	
	  	\draw ($(C1)+(90:0.2)$) edge[Hedge,blue4,closed,curve through={($(D1)+(90:0.2)$) ($(D1)+(45:0.2)$) ($(D1)+(360:0.2)$) ($(D2)+(180:0.5)$) ($(C3)+(315:0.2)$) ($(C3)+(270:0.2)$) ($(C3)+(180:0.2)$) ($(C3)+(135:0.2)$) ($(C2)+(360:0.5)$) ($(C1)+(180:0.2)$)}] ($(C1)+(135:0.2)$);
	  	
	  	\draw ($(B1)+(90:0.3)$) edge[Medge,closed,curve through={($(B1)+(30:0.2)$) ($(B3)+(360:0.25)$) ($(B5)+(320:0.2)$) ($(B5)+(270:0.3)$) ($(B5)+(210:0.2)$) ($(B3)+(180:0.25)$)}] ($(B1)+(150:0.2)$);
	  	\draw ($(D1)+(90:0.3)$) edge[Medge,closed,curve through={($(D1)+(30:0.2)$) ($(D2)+(360:0.25)$) ($(D3)+(320:0.2)$) ($(D3)+(270:0.3)$) ($(D3)+(210:0.2)$) ($(D2)+(180:0.25)$)}] ($(D1)+(150:0.2)$);
	  	
	  	\draw ($(C3)+(315:0.2)$) edge[Medge,closed,curve through={($(C3)+(360:0.2)$) ($(C2)+(360:0.2)$) ($(C2)+(45:0.2)$) ($(C2)+(90:0.2)$) ($(C3)+(135:1.2)$) ($(B6)+(180:0.2)$) ($(B6)+(225:0.2)$) ($(B6)+(270:0.2)$) }] ($(C3)+(270:0.2)$);
	  	\draw ($(B2)+(135:0.2)$) edge[Medge,closed,curve through={($(B2)+(180:0.2)$) ($(B4)+(180:0.2)$) ($(B4)+(235:0.2)$) ($(B4)+(270:0.2)$) ($(B2)+(315:1.2)$) ($(C1)+(360:0.2)$) ($(C1)+(45:0.2)$) ($(C1)+(90:0.2)$) }] ($(B2)+(90:0.2)$);
	  	\draw ($(A2)+(45:0.2)$) edge[Medge,closed,curve through={($(A2)+(360:0.2)$) ($(A4)+(360:0.2)$) ($(A4)+(315:0.2)$) ($(A4)+(270:0.2)$) ($(A2)+(225:1.2)$) ($(A1)+(180:0.2)$) ($(A1)+(135:0.2)$) ($(A1)+(90:0.2)$) }] ($(A2)+(90:0.2)$);
	  	\draw ($(A5)+(225:0.2)$) edge[Medge,closed,curve through={($(A5)+(180:0.2)$) ($(A3)+(180:0.2)$) ($(A3)+(135:0.2)$) ($(A3)+(90:0.2)$) ($(A5)+(45:1.2)$) ($(A6)+(360:0.2)$) ($(A6)+(315:0.2)$) ($(A6)+(270:0.2)$) }] ($(A5)+(270:0.2)$);
	  	
	  	\coordinate (Astart) at ($(B2)+(\Hdist/2,\Vdist/2)$);
	  	\coordinate (Aend) at ($(B6)-(\Hdist/2,{\Vdist*(2/3)})$);
	  	\node[CutPink] (ACutlabel) at ($(B6)-(\Hdist/2,\Vdist)$) {$\Cut{}{A}$};
	  	\def\CutACurve{%
	  		(Astart).. ($(B2)+(360:0.5)$) .. ($(B4)+(360:0.4)$) .. ($(B4)+(270:\Vdist/2)$) .. ($(B6)+(180:0.4)$) .. (Aend)
	  	}
	  	\draw[CutPink,line width=3pt,opacity=0.7,use Hobby shortcut, closed=false] \CutACurve;
	  	
	  	\node[CutPink] (Alabel) at ($(B1)+(-\Hdist/4,\Vdist/2)$) {$A$};
	  	\node[CutPink] (ACompllabel) at ($(D1)+(-\Hdist/4,\Vdist/2)$) {$\Compl{A}$};

	  	\node (HAC) at ($(A5)+(0,-3*\Vdist)$) {\begin{tikzpicture}[scale=0.9]
	  		\path[use as bounding box] (-\Hdist/2,\Vdist/2) rectangle (4.5*\Hdist,-2.5*\Vdist);
	  		\node at ($(A1)+(-{\Hdist*(2/3)},{\Vdist*(1/3)})$) {\textcolor{blue4}{$H_{\Compl{A}}$}};
		  	\node[vertex] (A1) at (0,0) {};
		  	\node[vertex] (A2) at ($(A1)+(\Hdist,0)$) {};
		  	\node[vertex] (A3) at ($(A1)+(0,-\Vdist)$) {};
		  	\node[vertex] (A4) at ($(A3)+(\Hdist,0)$) {};
		  	\node[vertex] (A5) at ($(A3)+(0,-\Vdist)$) {};
		  	\node[vertex] (A6) at ($(A5)+(\Hdist,0)$) {};
		  	
		  	\node[vertex] (B1) at ($(A2)+(\Hdist,0)$) {};
		  	\node[vertex] (B2) at ($(B1)+(\Hdist,0)$) {};
		  	\node[vertex] (B3) at ($(B1)+(0,-\Vdist)$) {};
		  	\node[vertex] (B4) at ($(B3)+(\Hdist,0)$) {};
		  	\node[vertex] (B5) at ($(B3)+(0,-\Vdist)$) {};
		  	
		  	\node[vertex,CutPink,label={[CutPinkText,label distance=0.8ex]5:$\Compl{a}$}] (aC) at ($(B4)+(0.7*\Hdist,0)$) {};
		  	
		  	\draw ($(A1)+(90:0.3)$) edge[Hedge,closed,curve through={($(A1)+(30:0.2)$) ($(A3)+(360:0.25)$) ($(A5)+(320:0.2)$) ($(A5)+(270:0.3)$) ($(A5)+(210:0.2)$) ($(A3)+(180:0.25)$)}] ($(A1)+(150:0.2)$);
		  	\draw ($(A2)+(90:0.3)$) edge[Hedge,blue3,closed,curve through={($(A2)+(30:0.2)$) ($(A4)+(360:0.25)$) ($(A6)+(320:0.2)$) ($(A6)+(270:0.3)$) ($(A6)+(210:0.2)$) ($(A4)+(180:0.25)$)}] ($(A2)+(150:0.2)$);
		  	\draw ($(A2)+(180:0.3)$) edge[Hedge,closed,curve through={($(A2)+(120:0.2)$) ($(B1)+(90:0.25)$) ($(B2)+(60:0.2)$) ($(B2)+(360:0.3)$) ($(B2)+(300:0.2)$) ($(B1)+(270:0.25)$)}] ($(A2)+(240:0.2)$);
		  	\draw ($(A4)+(180:0.3)$) edge[Hedge,closed,curve through={($(A4)+(120:0.2)$) ($(B3)+(90:0.25)$) ($(B4)+(60:0.2)$) ($(B4)+(360:0.3)$) ($(B4)+(300:0.2)$) ($(B3)+(270:0.25)$)}] ($(A4)+(240:0.2)$);
		  	\draw ($(A6)+(180:0.3)$) edge[Hedge,closed,curve through={($(A6)+(120:0.2)$) ($(B5)+(90:0.25)$) ($(B6)+(90:0.25)$) ($(aC)+(60:0.2)$) ($(aC)+(360:0.3)$) ($(aC)+(300:0.2)$) ($(B6)+(270:0.25)$) ($(B5)+(270:0.25)$)}] ($(A6)+(240:0.2)$);
		  	
		  	\draw ($(B1)+(90:0.3)$) edge[Medge,closed,curve through={($(B1)+(30:0.2)$) ($(B3)+(360:0.25)$) ($(B5)+(320:0.2)$) ($(B5)+(270:0.3)$) ($(B5)+(210:0.2)$) ($(B3)+(180:0.25)$)}] ($(B1)+(150:0.2)$);

		  	\draw ($(B2)+(135:0.2)$) edge[Medge,closed,curve through={($(B2)+(180:0.2)$) ($(B4)+(180:0.2)$) ($(B4)+(235:0.2)$) ($(B4)+(270:0.2)$) ($(B4)+(315:0.3)$) ($(aC)+(235:0.3)$) ($(aC)+(360:0.3)$) ($(aC)+(90:0.3)$) }] ($(B2)+(90:0.2)$);
		  	\draw ($(A2)+(45:0.2)$) edge[Medge,closed,curve through={($(A2)+(360:0.2)$) ($(A4)+(360:0.2)$) ($(A4)+(315:0.2)$) ($(A4)+(270:0.2)$) ($(A2)+(225:1.2)$) ($(A1)+(180:0.2)$) ($(A1)+(135:0.2)$) ($(A1)+(90:0.2)$) }] ($(A2)+(90:0.2)$);
		  	\draw ($(A5)+(225:0.2)$) edge[Medge,closed,curve through={($(A5)+(180:0.2)$) ($(A3)+(180:0.2)$) ($(A3)+(135:0.2)$) ($(A3)+(90:0.2)$) ($(A5)+(45:1.2)$) ($(A6)+(360:0.2)$) ($(A6)+(315:0.2)$) ($(A6)+(270:0.2)$) }] ($(A5)+(270:0.2)$);
		  	
		  	\coordinate (Bstart) at ($(A1)+(\Hdist/2,\Vdist/2)$);
		  	\coordinate (Bend) at ($(A6)-(0,{\Vdist*(2/3)})$);
		  	\node[CutRedText] (BCutlabel) at ($(A6)+(0,-\Vdist)$) {$\Cut{}{B}$};
		  	\def\CutBCurve{%
		  		(Bstart).. ($(A2)+(180:0.5)$) .. ($(A4)+(180:0.4)$) .. ($(A4)+(270:\Vdist/2)$) .. ($(A6)+(20:0.4)$) .. (Bend)
		  	}
		  	\draw[CutRed,line width=3pt,opacity=0.7,use Hobby shortcut, closed=false] \CutBCurve;
		  	
		  	\node[CutRedText] (Blabel) at ($(A1)+(0,\Vdist/2)$) {$B$};
		  	\node[CutRedText] (BCompllabel) at ($(B2)+(-\Hdist/4,\Vdist/2)$) {$\Compl{B}$};
	  	\end{tikzpicture}};
	  	
	  	\node (HA) at ($(C3)+(\Hdist,-3*\Vdist)$) {\begin{tikzpicture}[scale=0.9]
	  		\path[use as bounding box] (2.5*\Hdist,\Vdist/2) rectangle (5.5*\Hdist,-2.5*\Vdist);
	  		\node at ($(B2)$) {\textcolor{blue4}{$H_{A}$}};
		  	\node[vertex,CutPink,label={[CutPinkText,label distance=0.8ex]175:$a$}] (a) at ($(B3)+(\Hdist,0)$) {};
		  	\node[vertex] (B6) at ($(B5)+(\Hdist,0)$) {};
		  	
		  	\node[vertex] (C1) at ($(B2)+(\Hdist,0)$) {};
		  	\node[vertex] (D1) at ($(C1)+(\Hdist,0)$) {};
		  	\node[vertex] (C2) at ($(C1)+(0,-\Vdist)$) {};
		  	\node[vertex] (D2) at ($(C2)+(\Hdist,0)$) {};
		  	\node[vertex] (C3) at ($(C2)+(0,-\Vdist)$) {};
		  	\node[vertex] (D3) at ($(C3)+(\Hdist,0)$) {};
		  	
		  	\draw ($(C1)+(90:0.3)$) edge[Hedge,closed,curve through={($(C1)+(30:0.2)$) ($(C2)+(360:0.25)$) ($(C3)+(320:0.2)$) ($(C3)+(270:0.3)$) ($(C3)+(210:0.2)$) ($(C2)+(180:0.25)$)}] ($(C1)+(150:0.2)$);
		  	\draw ($(B6)+(180:0.2)$) edge[Hedge,closed,curve through={($(a)+(180:0.2)$) ($(a)+(90:0.25)$) ($(a)+(0:0.2)$) ($(B6)+(360:0.2)$)}] ($(B6)+(270:0.25)$);
		  	\draw ($(D2)+(45:0.2)$) edge[Hedge,blue3,closed,curve through={($(D2)+(360:0.2)$) ($(D3)+(360:0.2)$) ($(D3)+(315:0.2)$) ($(D3)+(270:0.2)$) ($(D2)+(225:1.2)$) ($(C2)+(180:0.2)$) ($(C2)+(135:0.2)$) ($(C2)+(90:0.2)$) }] ($(D2)+(90:0.2)$);
		  	\draw ($(C1)+(90:0.2)$) edge[Hedge,blue4,closed,curve through={($(D1)+(90:0.2)$) ($(D1)+(45:0.2)$) ($(D1)+(360:0.2)$) ($(D2)+(180:0.5)$) ($(C3)+(315:0.2)$) ($(C3)+(270:0.2)$) ($(C3)+(180:0.2)$) ($(C3)+(135:0.2)$) ($(C2)+(360:0.5)$) ($(C1)+(180:0.2)$)}] ($(C1)+(135:0.2)$);
		  	
		  	\draw ($(D1)+(90:0.3)$) edge[Medge,closed,curve through={($(D1)+(30:0.2)$) ($(D2)+(360:0.25)$) ($(D3)+(320:0.2)$) ($(D3)+(270:0.3)$) ($(D3)+(210:0.2)$) ($(D2)+(180:0.25)$)}] ($(D1)+(150:0.2)$);
		  	
		  	\draw ($(C3)+(315:0.2)$) edge[Medge,closed,curve through={($(C3)+(360:0.2)$) ($(C2)+(360:0.2)$) ($(C2)+(45:0.2)$) ($(C2)+(90:0.2)$) ($(C3)+(135:1.2)$) ($(B6)+(180:0.2)$) ($(B6)+(225:0.2)$) ($(B6)+(270:0.2)$) }] ($(C3)+(270:0.2)$);
		  	\draw ($(a)+(110:0.3)$) edge[Medge,closed,curve through={($(a)+(235:0.25)$) ($(a)+(340:0.3)$) ($(C1)+(290:0.3)$) ($(C1)+(45:0.25)$)}] ($(C1)+(150:0.3)$);
		  	
		  	\coordinate (Cstart) at ($(C1)+(\Hdist/2,\Vdist/2)$);
		  	\coordinate (Cend) at ($(C3)-(\Hdist/2,{\Vdist*(2/3)})$);
		  	\node[CutDarkGreenText] (CCutlabel) at ($(C3)-(\Hdist/2,\Vdist)$) {$\Cut{}{C}$};
		  	\def\CutCCurve{%
		  		(Cstart).. ($(C1)+(360:0.5)$) .. ($(C1)+(270:\Vdist/2)$) .. ($(C2)+(180:0.4)$) .. ($(C3)+(180:0.4)$) .. (Cend)
		  	}
		  	\draw[CutDarkGreen,line width=3pt,opacity=0.7,use Hobby shortcut, closed=false] \CutCCurve;
		  	
		  	\node[CutDarkGreenText] (Clabel) at ($(C1)+(-\Hdist/4,\Vdist/2)$) {$C$};
		  	\node[CutDarkGreenText] (CCompllabel) at ($(D1)+(\Hdist/4,\Vdist/2)$) {$\Compl{C}$};
	  	\end{tikzpicture}};
	  	
	  	\node (HB) at ($(HAC)+(\Hdist,-4*\Vdist)$) {\begin{tikzpicture}[scale=0.9]
		  	\path[use as bounding box] (0.5*\Hdist,\Vdist/2) rectangle (3.5*\Hdist,-2.5*\Vdist);
		  	\node[Brace] at ($(B1)+({\Hdist*(2/3)},{\Vdist*(2/3)})$) {\textcolor{blue4}{$H_{B}$}};
		  	\node (A1) at (0,0) {};
			\node[vertex] (A2) at ($(A1)+(\Hdist,0)$) {};
			\node[vertex] (A4) at ($(A2)+(0,-\Vdist)$) {};
			\node[vertex,CutRed,label={[CutRedText,label distance=0.8ex]185:$b$}] (A6) at ($(A4)+(0,-\Vdist)$) {};
			\node[vertex] (B1) at ($(A2)+(\Hdist,0)$) {};
			\node[vertex] (B2) at ($(B1)+(\Hdist,0)$) {};
			\node[vertex] (B3) at ($(B1)+(0,-\Vdist)$) {};
			\node[vertex] (B4) at ($(B3)+(\Hdist,0)$) {};
			\node[vertex] (B5) at ($(B3)+(0,-\Vdist)$) {};
			\node[vertex,CutPink,label={[CutPinkText,label distance=0.8ex]355:$\Compl{a}$}] (B6) at ($(B5)+(\Hdist,0)$) {};
					
			\draw ($(A2)+(180:0.3)$) edge[Hedge,closed,curve through={($(A2)+(120:0.2)$) ($(B1)+(90:0.25)$) ($(B2)+(60:0.2)$) ($(B2)+(360:0.3)$) ($(B2)+(300:0.2)$) ($(B1)+(270:0.25)$)}] ($(A2)+(240:0.2)$);
			\draw ($(A4)+(180:0.3)$) edge[Hedge,closed,curve through={($(A4)+(120:0.2)$) ($(B3)+(90:0.25)$) ($(B4)+(60:0.2)$) ($(B4)+(360:0.3)$) ($(B4)+(300:0.2)$) ($(B3)+(270:0.25)$)}] ($(A4)+(240:0.2)$);
			\draw ($(A6)+(180:0.3)$) edge[Hedge,closed,curve through={($(A6)+(120:0.2)$) ($(B5)+(90:0.25)$) ($(B6)+(60:0.2)$) ($(B6)+(360:0.3)$) ($(B6)+(300:0.2)$) ($(B5)+(270:0.25)$)}] ($(A6)+(240:0.2)$);
			\draw ($(A2)+(90:0.3)$) edge[Hedge,closed,curve through={($(A2)+(30:0.2)$) ($(A4)+(360:0.25)$) ($(A6)+(320:0.2)$) ($(A6)+(270:0.3)$) ($(A6)+(210:0.2)$) ($(A4)+(180:0.25)$)}] ($(A2)+(150:0.2)$);
			
			\draw ($(B1)+(90:0.3)$) edge[Medge,closed,curve through={($(B1)+(30:0.2)$) ($(B3)+(360:0.25)$) ($(B5)+(320:0.2)$) ($(B5)+(270:0.3)$) ($(B5)+(210:0.2)$) ($(B3)+(180:0.25)$)}] ($(B1)+(150:0.2)$);
			\draw ($(A2)+(90:0.35)$) edge[Medge,closed,curve through={($(A2)+(30:0.3)$) ($(A4)+(360:0.35)$) ($(A6)+(320:0.3)$) ($(A6)+(270:0.35)$) ($(A6)+(210:0.3)$) ($(A4)+(180:0.35)$)}] ($(A2)+(150:0.3)$);
			\draw ($(B2)+(90:0.3)$) edge[Medge,closed,curve through={($(B2)+(30:0.2)$) ($(B4)+(360:0.25)$) ($(B6)+(320:0.2)$) ($(B6)+(270:0.3)$) ($(B6)+(210:0.2)$) ($(B4)+(180:0.25)$)}] ($(B2)+(150:0.2)$);
		\end{tikzpicture}};
	  	
	  	\node (HBC) at ($(HAC)+(-2*\Hdist,-4*\Vdist)$) {\begin{tikzpicture}[scale=0.9]
	  		\path[use as bounding box] (-0.5*\Hdist,\Vdist/2) rectangle (1.5*\Hdist,-2.5*\Vdist);
	  		\node at ($(A1)+({\Hdist*(1/3)},{\Vdist*(2/3)})$) {\textcolor{blue4}{$H_{\Compl{B}}$}};
		  	\node[vertex] (A1) at (0,0) {};
		  	\node[vertex,CutRed,label={[CutRedText,label distance=0.8ex]5:$\Compl{b}$}] (bc) at ($(A1)+(\Hdist,-\Vdist/2)$) {};
		  	\node[vertex] (A3) at ($(A1)+(0,-\Vdist)$) {};
		  	\node[vertex] (A5) at ($(A3)+(0,-\Vdist)$) {};
		  	\node[vertex] (A6) at ($(A5)+(\Hdist,0)$) {};
		  	
		  	\draw ($(A1)+(90:0.3)$) edge[Hedge,closed,curve through={($(A1)+(30:0.2)$) ($(A3)+(360:0.25)$) ($(A5)+(320:0.2)$) ($(A5)+(270:0.3)$) ($(A5)+(210:0.2)$) ($(A3)+(180:0.25)$)}] ($(A1)+(150:0.2)$);
		  	\draw ($(A6)+(180:0.2)$) edge[Hedge,closed,curve through={($(bc)+(180:0.2)$) ($(bc)+(90:0.25)$) ($(bc)+(0:0.2)$) ($(A6)+(360:0.2)$)}] ($(A6)+(270:0.25)$);
		  	\draw ($(A6)+(180:0.3)$) edge[Hedge,blue3,closed,curve through={($(bc)+(180:0.3)$) ($(bc)+(90:0.35)$) ($(bc)+(0:0.3)$) ($(A6)+(360:0.3)$)}] ($(A6)+(270:0.35)$);
		  	
		  	\draw ($(A1)+(135:0.3)$) edge[Medge,closed,curve through={($(A1)+(70:0.25)$) ($(bc)+(110:0.25)$) ($(bc)+(315:0.3)$) ($(bc)+(200:0.3)$)}] ($(A1)+(250:0.2)$);
		  	\draw ($(A5)+(225:0.2)$) edge[Medge,closed,curve through={($(A5)+(180:0.2)$) ($(A3)+(180:0.2)$) ($(A3)+(135:0.2)$) ($(A3)+(90:0.2)$) ($(A5)+(45:1.2)$) ($(A6)+(360:0.2)$) ($(A6)+(315:0.2)$) ($(A6)+(270:0.2)$) }] ($(A5)+(270:0.2)$);
		  	
		  	
		  	\coordinate (Dstart) at ($(A1)-(\Hdist/2,\Vdist/2)$);
		  	\coordinate (Dend) at ($(A5)+(\Hdist/2,-{\Vdist*(2/3)})$);
		  	\node[CutLightGreenText] (DCutlabel) at ($(A5)+(0.7*\Hdist,-\Vdist)$) {$\Cut{}{D}$};
		  	\def\CutDCurve{%
		  		(Dstart).. ($(A3)+(90:0.5)$) .. ($(A3)+(0:0.6)$) .. ($(A5)+(20:0.4)$) .. (Dend)
		  	}
		  	\draw[CutLightGreen,line width=3pt,opacity=0.7,use Hobby shortcut, closed=false] \CutDCurve;
		  	
		  	\node[CutLightGreenText] (Dlabel) at ($(A5)+(-\Hdist/4,-\Vdist/2)$) {$D$};
		  	\node[CutLightGreenText] (DCompllabel) at ($(A6)+(\Hdist/4,-\Vdist/2)$) {$\Compl{D}$};
	  	\end{tikzpicture}};
	  	
	  	\node (HDC) at ($(HBC)+(0,-4*\Vdist)$) {\begin{tikzpicture}[scale=0.9]
	  		\path[use as bounding box] (-0.5*\Hdist,\Vdist/2) rectangle (0.5*\Hdist,-2.5*\Vdist);
	  		\node[Brace] at ($(A1)+(-{\Hdist*(2/3)},{\Vdist*(1/3)})$) {\textcolor{blue4}{$H_{\Compl{D}}$}};
		  	\node[vertex,CutLightGreen,label={[CutLightGreenText,label distance=0.8ex]5:$\Compl{d}$}] (A1) at (0,0) {};
		  	\node[vertex] (A3) at ($(A1)+(0,-\Vdist)$) {};
		  	\node[vertex] (A5) at ($(A3)+(0,-\Vdist)$) {};
		  	
		  	\draw ($(A1)+(90:0.3)$) edge[Hedge,closed,curve through={($(A1)+(30:0.2)$) ($(A3)+(360:0.25)$) ($(A5)+(320:0.2)$) ($(A5)+(270:0.3)$) ($(A5)+(210:0.2)$) ($(A3)+(180:0.25)$)}] ($(A1)+(150:0.2)$);
		  	
		  	\draw ($(A1)+(90:0.4)$) edge[Medge,closed,curve through={($(A1)+(30:0.3)$) ($(A3)+(360:0.35)$) ($(A5)+(320:0.3)$) ($(A5)+(270:0.4)$) ($(A5)+(210:0.3)$) ($(A3)+(180:0.35)$)}] ($(A1)+(150:0.3)$);
		  
	  	\end{tikzpicture}};
	  	
	  	\node (HD) at ($(HBC)+(2*\Hdist,-4*\Vdist)$) {\begin{tikzpicture}[scale=0.9]
	  		\path[use as bounding box] (-0.5*\Hdist,\Vdist/2) rectangle (1.5*\Hdist,-2.5*\Vdist);
	  		\node[Brace] at ($(A1)+(\Hdist,{\Vdist*(1/3)})$) {\textcolor{blue4}{$H_{D}$}};
		  	\node[vertex] (A1) at (\Hdist/2,0) {};
		  	\node[vertex,CutRed,label={[CutRedText,label distance=1ex]5:$\Compl{b}$}] (bc) at ($(A1)+(\Hdist/2,-\Vdist)$) {};
		  	\node[vertex,CutLightGreen,label={[CutLightGreenText,label distance=1ex]175:$d$}] (d) at ($(A1)+(-\Hdist/2,-\Vdist)$) {};
		  	\node[vertex] (A6) at ($(A1)+(0,-2*\Vdist)$) {};

		  	\draw ($(A1)+(90:0.3)$) edge[Hedge,closed,curve through={($(A1)+(30:0.2)$) ($(d)+(320:0.2)$) ($(d)+(270:0.3)$) ($(d)+(210:0.2)$)}] ($(A1)+(150:0.2)$);

		  	\draw ($(A6)+(180:0.2)$) edge[Hedge,closed,curve through={($(bc)+(180:0.2)$) ($(bc)+(90:0.25)$) ($(bc)+(0:0.2)$) ($(A6)+(360:0.2)$)}] ($(A6)+(270:0.25)$);
		  	\draw ($(A6)+(180:0.3)$) edge[Hedge,blue3,closed,curve through={($(bc)+(180:0.3)$) ($(bc)+(90:0.35)$) ($(bc)+(0:0.3)$) ($(A6)+(360:0.3)$)}] ($(A6)+(270:0.35)$);
		  	
		  	\draw ($(bc)+(50:0.2)$) edge[Medge,closed,curve through={($(bc)+(335:0.3)$) ($(bc)+(210:0.2)$) ($(A1)+(240:0.2)$) ($(A1)+(135:0.3)$) }] ($(A1)+(70:0.2)$);
		  	\draw ($(A6)+(50:0.2)$) edge[Medge,closed,curve through={($(A6)+(335:0.3)$) ($(A6)+(210:0.2)$) ($(d)+(240:0.2)$) ($(d)+(135:0.3)$) }] ($(d)+(70:0.2)$);
	  	\end{tikzpicture}};
	  	
	  	\node (HCC) at ($(HA)+(-\Hdist,-4*\Vdist)$) {\begin{tikzpicture}[scale=0.9]
		  	\path[use as bounding box] (2.5*\Hdist,\Vdist/2) rectangle (4.5*\Hdist,-2.5*\Vdist);
	  		\node[Brace] at ($(B2)+(0,{\Vdist*(1/3)})$) {\textcolor{blue4}{$H_{\Compl{C}}$}};
		  	\node[vertex,CutPink,label={[CutPinkText,label distance=0.8ex]175:$a$}] (a) at ($(B3)+(\Hdist,0)$) {};
		  	\node[vertex] (B6) at ($(a)+(\Hdist/2,-\Vdist)$) {};
		  	\node[vertex] (C1) at ($(B2)+(\Hdist/2,0)$) {};
		  	\node[vertex,CutDarkGreen,label={[CutDarkGreenText,label distance=0.8ex]5:$\Compl{c}$}] (cC) at ($(C1)+(\Hdist/2,-\Vdist)$) {};
		  	
		  	\draw ($(B6)+(180:0.2)$) edge[Hedge,closed,curve through={($(a)+(180:0.2)$) ($(a)+(90:0.25)$) ($(a)+(0:0.2)$) ($(B6)+(360:0.2)$)}] ($(B6)+(270:0.25)$);
		  	\draw ($(cC)+(180:0.2)$) edge[Hedge,closed,curve through={($(C1)+(180:0.2)$) ($(C1)+(90:0.25)$) ($(C1)+(0:0.2)$) ($(cC)+(360:0.2)$)}] ($(cC)+(270:0.25)$);
		  	
		  	\draw ($(C1)+(90:0.3)$) edge[Medge,closed,curve through={($(C1)+(30:0.2)$) ($(a)+(320:0.2)$) ($(a)+(270:0.3)$) ($(a)+(210:0.2)$)}] ($(C1)+(150:0.2)$);
		  	\draw ($(cC)+(90:0.3)$) edge[Medge,closed,curve through={($(cC)+(30:0.2)$) ($(B6)+(320:0.2)$) ($(B6)+(270:0.3)$) ($(B6)+(210:0.2)$)}] ($(cC)+(150:0.2)$);		  	
		\end{tikzpicture}};
		
	  	\node (HC) at ($(HA)+(1.5*\Hdist,-4*\Vdist)$) {\begin{tikzpicture}[scale=0.9]
		  	\path[use as bounding box] (3*\Hdist,\Vdist/2) rectangle (5*\Hdist,-2.5*\Vdist);
	  		\node at ($(D1)+({\Hdist*(1/4)},{\Vdist*(1/3)})$) {\textcolor{blue4}{$H_{C}$}};
		  	\node[vertex,CutDarkGreen,label={[CutDarkGreenText,label distance=1ex]175:$c$}] (c) at ($(C1)$) {};
		  	\node[vertex] (D1) at ($(C1)+(\Hdist,0)$) {};
		  	\node[vertex] (C2) at ($(C1)+(0,-\Vdist)$) {};
		  	\node[vertex] (D2) at ($(C2)+(\Hdist,0)$) {};
		  	\node[vertex] (C3) at ($(C2)+(0,-\Vdist)$) {};
		  	\node[vertex] (D3) at ($(C3)+(\Hdist,0)$) {};
		  	
		  	\draw ($(C1)+(90:0.3)$) edge[Hedge,closed,curve through={($(C1)+(30:0.2)$) ($(C2)+(360:0.25)$) ($(C3)+(320:0.2)$) ($(C3)+(270:0.3)$) ($(C3)+(210:0.2)$) ($(C2)+(180:0.25)$)}] ($(C1)+(150:0.2)$);
		  	
		  	\draw ($(D2)+(45:0.2)$) edge[Hedge,blue3,closed,curve through={($(D2)+(360:0.2)$) ($(D3)+(360:0.2)$) ($(D3)+(315:0.2)$) ($(D3)+(270:0.2)$) ($(D2)+(225:1.2)$) ($(C2)+(180:0.2)$) ($(C2)+(135:0.2)$) ($(C2)+(90:0.2)$) }] ($(D2)+(90:0.2)$);
		  	\draw ($(C1)+(90:0.2)$) edge[Hedge,blue4,closed,curve through={($(D1)+(90:0.2)$) ($(D1)+(45:0.2)$) ($(D1)+(360:0.2)$) ($(D2)+(180:0.5)$) ($(C3)+(315:0.2)$) ($(C3)+(270:0.2)$) ($(C3)+(180:0.2)$) ($(C3)+(135:0.2)$) ($(C2)+(360:0.5)$) ($(C1)+(180:0.2)$)}] ($(C1)+(135:0.2)$);
		  	
		  	\draw ($(C1)+(90:0.35)$) edge[Medge,closed,curve through={($(C1)+(30:0.3)$) ($(C2)+(360:0.35)$) ($(C3)+(320:0.3)$) ($(C3)+(270:0.35)$) ($(C3)+(210:0.3)$) ($(C2)+(180:0.35)$)}] ($(C1)+(150:0.3)$);
		  	\draw ($(D1)+(90:0.3)$) edge[Medge,closed,curve through={($(D1)+(30:0.2)$) ($(D2)+(360:0.25)$) ($(D3)+(320:0.2)$) ($(D3)+(270:0.3)$) ($(D3)+(210:0.2)$) ($(D2)+(180:0.25)$)}] ($(D1)+(150:0.2)$);
		  	
		  	
		  	\coordinate (Estart) at ($(D2)+(\Hdist/2,\Vdist/2)$);
		  	\coordinate (Eend) at ($(C3)+(\Hdist/2,-{\Vdist*(2/3)})$);
		  	\node[CutVioletText] (ECutlabel) at ($(C3)+(\Hdist/4,-\Vdist)$) {$\Cut{}{E}$};
		  	\def\CutECurve{%
		  		(Estart).. ($(D2)+(90:\Vdist/2)$) .. ($(D2)+(180:\Hdist/2)$) .. ($(C3)+(360:\Hdist/2)$) .. (Eend)
		  	}
		  	\draw[CutViolet,line width=3pt,opacity=0.7,use Hobby shortcut, closed=false] \CutECurve;
		  	
		  	\node[CutVioletText] (Elabel) at ($(C3)+(-\Hdist/4,-\Vdist/2)$) {$E$};
		  	\node[CutVioletText] (ECompllabel) at ($(D3)+(\Hdist/4,-\Vdist/2)$) {$\Compl{E}$};
		\end{tikzpicture}};
		
		\node (HEC) at ($(HC)+(-2*\Hdist,-4*\Vdist)$) {\begin{tikzpicture}[scale=0.9]
		  	\path[use as bounding box] (3*\Hdist,\Vdist/2) rectangle (5*\Hdist,-2.5*\Vdist);
		  	\node[Brace] at ($(C1)+({\Hdist*(2/3)},{\Vdist*(1/3)})$) {\textcolor{blue4}{$H_{\Compl{E}}$}};
		  	\node[vertex,CutDarkGreen,label={[CutDarkGreenText,label distance=0.8ex]175:$c$}] (c) at ($(C1)$) {};
		  	\node[vertex] (D1) at ($(C1)+(\Hdist,-\Vdist/2)$) {};
		  	\node[vertex] (C2) at ($(C1)+(0,-\Vdist)$) {};
		  	\node[vertex] (C3) at ($(C2)+(0,-\Vdist)$) {};
		  	\node[vertex,CutViolet,label={[CutVioletText,label distance=0.8ex]355:$\Compl{e}$}] (e) at ($(C3)+(\Hdist,\Vdist/2)$) {};
		  	
		  	\draw ($(C1)+(90:0.3)$) edge[Hedge,closed,curve through={($(C1)+(30:0.2)$) ($(C2)+(360:0.25)$) ($(C3)+(320:0.2)$) ($(C3)+(270:0.3)$) ($(C3)+(210:0.2)$) ($(C2)+(180:0.25)$)}] ($(C1)+(150:0.2)$);
		  	
		  	\draw ($(e)+(110:0.3)$) edge[Hedge,blue3,closed,curve through={($(e)+(315:0.3)$) ($(e)+(200:0.3)$) ($(C2)+(260:0.3)$) ($(C2)+(135:0.3)$)}] ($(C2)+(70:0.2)$);
		  	\draw ($(C1)+(90:0.2)$) edge[Hedge,blue4,closed,curve through={($(D1)+(90:0.2)$) ($(D1)+(45:0.2)$) ($(D1)+(360:0.2)$) ($(D2)+(180:0.5)$) ($(C3)+(315:0.2)$) ($(C3)+(270:0.2)$) ($(C3)+(180:0.2)$) ($(C3)+(135:0.2)$) ($(C2)+(360:0.4)$) ($(C1)+(180:0.2)$)}] ($(C1)+(135:0.2)$);
		  	
		  	\draw ($(C1)+(90:0.35)$) edge[Medge,closed,curve through={($(C1)+(30:0.3)$) ($(C2)+(360:0.35)$) ($(C3)+(320:0.3)$) ($(C3)+(270:0.35)$) ($(C3)+(210:0.3)$) ($(C2)+(180:0.35)$)}] ($(C1)+(150:0.3)$);
		  	\draw ($(D1)+(90:0.3)$) edge[Medge,closed,curve through={($(D1)+(30:0.2)$) ($(D2)+(360:0.25)$) ($(e)+(320:0.2)$) ($(e)+(270:0.3)$) ($(e)+(210:0.2)$) ($(D2)+(180:0.25)$)}] ($(D1)+(150:0.2)$);
		\end{tikzpicture}};
		
		\node (HE) at ($(HC)+(0,-4*\Vdist)$) {\begin{tikzpicture}[scale=0.9]
		  	\path[use as bounding box] (4*\Hdist,\Vdist/2) rectangle (5*\Hdist,-2.5*\Vdist);
	  		\node[Brace] at ($(D1)+({\Hdist*(2/3)},{\Vdist*(1/3)})$) {\textcolor{blue4}{$H_{E}$}};
		  	\node[vertex,CutViolet,label={[CutVioletText,label distance=0.8ex]5:$e$}] (e) at ($(C1)+(\Hdist,0)$) {};
		  	\node[vertex] (D2) at ($(C2)+(\Hdist,0)$) {};
		  	\node[vertex] (D3) at ($(C3)+(\Hdist,0)$) {};
		  	
		  	\draw ($(e)+(90:0.3)$) edge[Hedge,closed,curve through={($(e)+(30:0.2)$) ($(D2)+(360:0.25)$) ($(D3)+(320:0.2)$) ($(D3)+(270:0.3)$) ($(D3)+(210:0.2)$) ($(D2)+(180:0.25)$)}] ($(e)+(150:0.2)$);
		  	
		  	\draw ($(e)+(90:0.35)$) edge[Medge,closed,curve through={($(e)+(30:0.3)$) ($(D2)+(360:0.35)$) ($(D3)+(320:0.3)$) ($(D3)+(270:0.35)$) ($(D3)+(210:0.3)$) ($(D2)+(180:0.35)$)}] ($(e)+(150:0.3)$);
		\end{tikzpicture}};
	  	
	  	\draw[thick,->,-latex,shorten <=4mm] (B5) to (HAC);
	  	\draw[thick,->,-latex,shorten <=4mm] ($(C3)-(\Hdist/2,0)$) to (HA);
	  	\draw[thick,->,-latex,shorten <=20mm] ($(HAC)+(-\Hdist,0)$) to (HBC);
	  	\draw[thick,->,-latex,shorten <=2mm] (HAC) to (HB);
	  	\draw[thick,->,-latex,shorten <=20mm] ($(HA)+(-\Hdist,0)$) to (HCC);
	  	\draw[thick,->,-latex,shorten <=2mm] (HA) to (HC);
	  	\draw[thick,->,-latex,shorten <=18mm] ($(HBC)+(-\Hdist/2,0)$) to (HDC);
	  	\draw[thick,->,-latex,shorten <=2mm] (HBC) to (HD);
	  	\draw[thick,->,-latex,shorten <=2mm] (HC) to (HEC);
	  	\draw[thick,->,-latex,shorten <=18mm] ($(HC)+(\Hdist/2,0)$) to (HE);
  	\end{tikzpicture}
  	\caption{A Tight Cut Decomposition of a matching covered hypergraph.}
    \label{fig:tcd}
\end{figure}
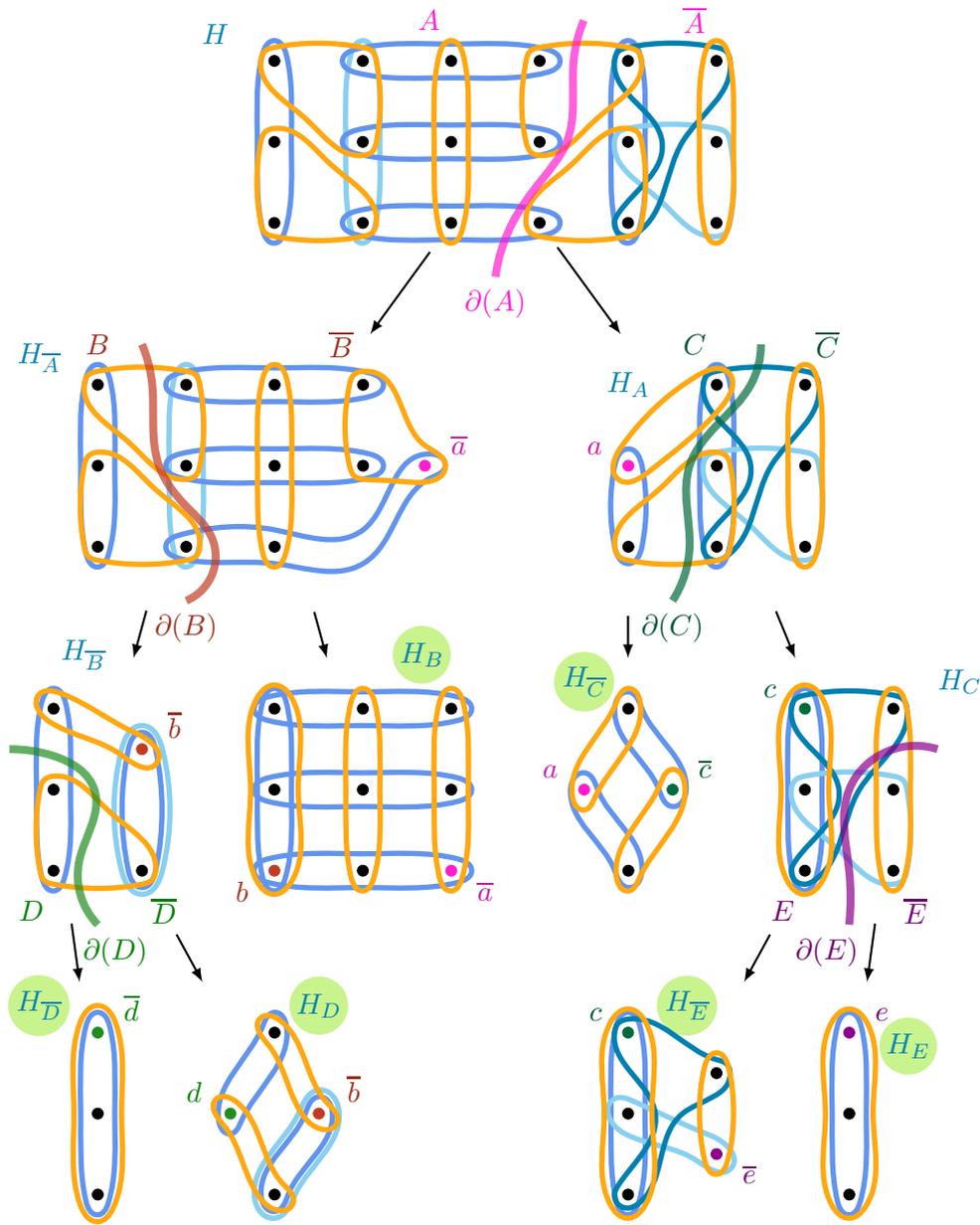  

  We do not specify the concrete choice of the tight cut in \cref{chooseTightcut} of \cref{alg:tcd}.
  Thus, different runs might give different outputs.
  For a hypergraph $H$ we denote the set of tight cut decompositions by $\Fkt{\mathcal{TD}}{H}$.
  We say that two tight cut decompositions $H_1,\ldots, H_i$ and $H'_1,\ldots, H'_j$ are equivalent if and only if there exists a bijection $\phi : [i]\to [j]$ such that $H_{s}$ and $H'_{\Fkt{\phi}{s}}$ are isomorphic up to parallel edges.
  For graphs any two tight cut decompositions are equivalent in this sense which was shown by Lov{\'a}sz \cite{lovasz1987}.
  
  \begin{figure}[t]
  	\centering
  	\begin{tikzpicture}[scale=0.86]
  	\definecolor{CutViolet}{rgb}{0.73, 0.2, 0.52}
  	\definecolor{CutGreen}{rgb}{0.13, 0.55, 0.13}
  	\definecolor{blue3}{rgb}{0.54, 0.81, 0.94}
  	\definecolor{blue2}{rgb}{0.0, 0.5, 1.0} 
  	\definecolor{blue4}{rgb}{0.0, 0.5, 0.69}
  	\node (M) at (0,0) {\begin{tikzpicture}
  		\node[vertex] (x2) at (0,0) {};
  		\node[vertex] (c1) at ($(x2)+({360/5*3}:1)$) {};
  		\node[vertex] (u1) at ($(c1)+({360/5*1.5}:1)$) {};
  		\node[vertex] (u2) at ($(c1)+({360/5*2.5}:1)$) {};
  		\node[vertex] (u3) at ($(c1)+({360/5*3.5}:1)$) {};
  		\node[vertex] (x1) at ($(c1)+({360/5*4.5}:1)$) {};
  		\node[vertex] (c2) at ($(x2)+({360/5*4.5}:1)$) {};
  		\node[vertex] (w1) at ($(c2)+({360/5*1}:1)$) {};
  		\node[vertex] (w2) at ($(c2)+({360/5*0}:1)$) {};
  		\node[vertex] (w3) at ($(c2)+({360/5*4}:1)$) {};
  		\tikzset{HSedge/.style={line width=2pt,draw,blue3}}
  		\draw[HSedge] (c1) to (u1);
  		\draw[HSedge] (c1) to (u2);
  		\draw[HSedge] (c1) to (u3);
  		\draw[HSedge] (c1) to (x1);
  		\draw[HSedge] (c1) to (x2);
  		\draw[HSedge] (c2) to (w1);
  		\draw[HSedge] (c2) to (w2);
  		\draw[HSedge] (c2) to (w3);
  		\draw[HSedge] (c2) to (x1);
  		\draw[HSedge] (c2) to (x2);
  		\draw[HSedge] (x2) to (u1);
  		\draw[HSedge] (u1) to (u2);
  		\draw[HSedge] (u2) to (u3);
  		\draw[HSedge] (u3) to (x1);
  		\draw[HSedge] (x1) to (x2);
  		\draw[HSedge] (x2) to (w1);
  		\draw[HSedge] (w1) to (w2);
  		\draw[HSedge] (w2) to (w3);
  		\draw[HSedge] (w3) to (x1);
  		\draw ($(x2)+(90:0.2)$) edge[Hedge,blue2,closed,curve through={($(x2)+(180:0.2)$) ($(x2)+(270:0.2)$)}] ($(x2)+(360:0.2)$);
  		\draw ($(c1)+(90:0.2)$) edge[Hedge,closed,curve through={($(c1)+(180:0.2)$) ($(u3)+(180:0.2)$) ($(x1)+(270:0.2)$) ($(w3)+(270:0.2)$) ($(w3)+(360:0.2)$) ($(w2)+(360:0.2)$) ($(c2)+(36:1.2)$) ($(w1)+(90:0.2)$) ($(c2)+(144:0.2)$) ($(x2)+(-72:0.9)$) ($(x1)+(90:0.2)$)}] ($(c1)+(45:0.2)$);
  		\definecolor{CutViolet}{rgb}{0.73, 0.2, 0.52}
  		\definecolor{CutGreen}{rgb}{0.13, 0.55, 0.13}
  		
  		\coordinate (Sstart) at ($(x2)+(50:0.7)$);
  		\coordinate (Send) at ($(u3)+({72*4.5}:0.5)$);
  		\node[CutGreen] (SCutlabel) at ($(x2)+(60:0.9)$) {$\Cut{}{S}$};
  		\def\CutSCurve{%
  			(Sstart).. ($(x2)+(36:0.5)$) .. ($(c2)+(144:0.5)$) .. ($(x2)+(270:0.5)$) .. ($(c1)+({72*4.5}:0.5)$) .. ($(c1)+({72*4}:1)$) .. (Send)
  		}
  		\draw[CutGreen,line width=3pt,opacity=0.7,use Hobby shortcut, closed=false] \CutSCurve;
  		
  		\coordinate (Tstart) at ($(x2)+(130:0.7)$);
  		\coordinate (Tend) at ($(w3)+({72*3}:0.5)$);
  		\node[CutViolet] (TCutlabel) at ($(x2)+(120:0.9)$) {$\Cut{}{T}$};
  		\def\CutTCurve{%
  			(Tstart).. ($(x2)+(144:0.5)$) .. ($(c1)+(36:0.5)$) .. ($(x2)+(270:0.5)$) .. ($(c2)+({72*3}:0.5)$) .. ($(c2)+({72*3.5}:1)$) .. (Tend)
  		}
  		\draw[CutViolet,line width=3pt,opacity=0.7,use Hobby shortcut, closed=false] \CutTCurve;
  		\node[fill=white,circle,opacity=0.6,inner sep=2.7pt] (x2label) at ($(x2)+(135:0.2)$) {};
  		\node[fill=white,circle,opacity=0.6,inner sep=2.7pt] (x1label) at ($(x1)+(135:0.2)$) {};
  		\node at ($(x2)+(135:0.2)$) {$v$};
  		\node at ($(x1)+(135:0.2)$) {$w$};
  		\end{tikzpicture}};
  	\node (HSQ) at ($(M)+(-5,-2)$) {\begin{tikzpicture}
  		\node[vertex] (x2) at (0,0) {}; 
  		\node[vertex] (c1) at ($(x2)+({360/5*3}:1)$) {};
  		\node[vertex] (u1) at ($(c1)+({360/5*1.5}:1)$) {};
  		\node[vertex] (u2) at ($(c1)+({360/5*2.5}:1)$) {};
  		\node[vertex] (u3) at ($(c1)+({360/5*3.5}:1)$) {};
  		\node[vertex,CutGreen] (x1) at ($(c1)+({360/5*4.5}:1)$) {}; 
  		\tikzset{HSedge/.style={line width=2pt,draw,blue3}}
  		\draw[HSedge] (c1) to (u1);
  		\draw[HSedge] (c1) to (u2);
  		\draw[HSedge] (c1) to (u3);
  		\draw[HSedge] (c1) to (x1);
  		\draw[HSedge] (c1) to (x2);
  		\draw[HSedge] (x2) to (u1);
  		\draw[HSedge] (u1) to (u2);
  		\draw[HSedge] (u2) to (u3);
  		\draw[HSedge] (u3) to (x1);
  		\draw[HSedge] (x1) to (x2);
  		\draw ($(x2)+(90:0.2)$) edge[Hedge,blue2,closed,curve through={($(x2)+(180:0.2)$) ($(x2)+(270:0.2)$)}] ($(x2)+(360:0.2)$);
  		\draw ($(c1)+(72:0.2)$) edge[Hedge,closed,curve through={($(x1)+(72:0.2)$) ($(x1)+(360:0.2)$) ($(x1)+(270:0.2)$) ($(u3)+(270:0.2)$) ($(u3)+(180:0.2)$) ($(u3)+(144:0.2)$)}] ($(c1)+(144:0.2)$);
  		
  		\node[fill=white,circle,opacity=0.6,inner sep=2.7pt] (x2label) at ($(x2)+(135:0.2)$) {};
  		\node at ($(x2)+(135:0.2)$) {$v$};
  		\node (Label) at ($(c1)+(135:1.5)$) {$H_{\Compl{S}}$};
  		\end{tikzpicture}};
  	\node (HT) at ($(M)+(-5,2)$) {\begin{tikzpicture}
  		\node[vertex,CutViolet] (x2) at (0,0) {}; 
  		\node[vertex] (c1) at ($(x2)+({360/5*3}:1)$) {};
  		\node[vertex] (u1) at ($(c1)+({360/5*1.5}:1)$) {};
  		\node[vertex] (u2) at ($(c1)+({360/5*2.5}:1)$) {};
  		\node[vertex] (u3) at ($(c1)+({360/5*3.5}:1)$) {};
  		\node[vertex] (x1) at ($(c1)+({360/5*4.5}:1)$) {}; 
  		\tikzset{HSedge/.style={line width=2pt,draw,blue3}}
  		\draw[HSedge] (c1) to (u1);
  		\draw[HSedge] (c1) to (u2);
  		\draw[HSedge] (c1) to (u3);
  		\draw[HSedge] (c1) to (x1);
  		\draw[HSedge] (c1) to (x2);
  		\draw[HSedge] (x2) to (u1);
  		\draw[HSedge] (u1) to (u2);
  		\draw[HSedge] (u2) to (u3);
  		\draw[HSedge] (u3) to (x1);
  		\draw[HSedge] (x1) to (x2);
  		\draw ($(c1)+(90:0.2)$) edge[Hedge,closed,curve through={($(c1)+(180:0.2)$) ($(u3)+(180:0.2)$) ($(u3)+(270:0.2)$) ($(x1)+(270:0.2)$) ($(x1)+(360:0.2)$) ($(x2)+(360:0.2)$)}] ($(x2)+(90:0.2)$);
  		
  		\node[fill=white,circle,opacity=0.6,inner sep=2.7pt] (x1label) at ($(x1)+(135:0.2)$) {};
  		\node at ($(x1)+(135:0.2)$) {$w$};
  		\node (Label) at ($(c1)+(135:1.5)$) {$H_{T}$};
  		\end{tikzpicture}};
  	\node (HS) at ($(M)+(5,-2)$) {\begin{tikzpicture}
  		\node[vertex,CutGreen] (x2) at (0,0) {}; 
  		\node[vertex] (x1) at ($(c1)+({360/5*4.5}:1)$) {}; 
  		\node[vertex] (c2) at ($(x2)+({360/5*4.5}:1)$) {};
  		\node[vertex] (w1) at ($(c2)+({360/5*1}:1)$) {};
  		\node[vertex] (w2) at ($(c2)+({360/5*0}:1)$) {};
  		\node[vertex] (w3) at ($(c2)+({360/5*4}:1)$) {};
  		\tikzset{HSedge/.style={line width=2pt,draw,blue3}}
  		\draw[HSedge] (c2) to (w1);
  		\draw[HSedge] (c2) to (w2);
  		\draw[HSedge] (c2) to (w3);
  		\draw[HSedge] (c2) to (x1);
  		\draw[HSedge] (c2) to (x2);
  		\draw[HSedge] (x1) to (x2);
  		\draw[HSedge] (x2) to (w1);
  		\draw[HSedge] (w1) to (w2);
  		\draw[HSedge] (w2) to (w3);
  		\draw[HSedge] (w3) to (x1);
  		\draw ($(x2)+({72*2}:0.2)$) edge[Hedge,closed,curve through={ ($(c2)+({72*2.5}:1.1)$) ($(x1)+({72*3}:0.2)$) ($(c2)+({72*3.5}:1.1)$) ($(w3)+({72*4}:0.2)$) ($(c2)+({72*4.5}:1.1)$) ($(w2)+({72*0}:0.2)$) ($(c2)+({72*0.5}:1.1)$) ($(w1)+({72*1}:0.2)$)}] ($(c2)+({72*1.5}:1.1)$);
  		
  		\node[fill=white,circle,opacity=0.6,inner sep=2.7pt] (x1label) at ($(x1)+(135:0.2)$) {};
  		\node at ($(x1)+(135:0.2)$) {$w$};
  		\node (Label) at ($(c2)+(45:1.5)$) {$H_{S}$};
  		\end{tikzpicture}};
  	\node (HTQ) at ($(M)+(5,2)$) {\begin{tikzpicture}
  		\node[vertex] (x2) at (0,0) {}; 
  		\node[vertex,CutViolet] (x1) at ($(c1)+({360/5*4.5}:1)$) {}; 
  		\node[vertex] (c2) at ($(x2)+({360/5*4.5}:1)$) {};
  		\node[vertex] (w1) at ($(c2)+({360/5*1}:1)$) {};
  		\node[vertex] (w2) at ($(c2)+({360/5*0}:1)$) {};
  		\node[vertex] (w3) at ($(c2)+({360/5*4}:1)$) {};
  		\tikzset{HSedge/.style={line width=2pt,draw,blue3}}
  		\draw[HSedge] (c2) to (w1);
  		\draw[HSedge] (c2) to (w2);
  		\draw[HSedge] (c2) to (w3);
  		\draw[HSedge] (c2) to (x1);
  		\draw[HSedge] (c2) to (x2);
  		\draw[HSedge] (x1) to (x2);
  		\draw[HSedge] (x2) to (w1);
  		\draw[HSedge] (w1) to (w2);
  		\draw[HSedge] (w2) to (w3);
  		\draw[HSedge] (w3) to (x1);
  		\draw ($(x2)+(90:0.2)$) edge[Hedge,blue2,closed,curve through={($(x2)+(180:0.2)$) ($(x2)+(270:0.2)$)}] ($(x2)+(360:0.2)$);
  		\draw ($(c2)+({72*2}:0.2)$) edge[Hedge,closed,curve through={ ($(c2)+({72*2.5}:0.5)$) ($(x1)+({72*3}:0.2)$) ($(c2)+({72*3.5}:1.1)$) ($(w3)+({72*4}:0.2)$) ($(c2)+({72*4.5}:1.1)$) ($(w2)+({72*0}:0.2)$) ($(c2)+({72*0.5}:1.1)$) ($(w1)+({72*1}:0.2)$)}] ($(c2)+({72*1.5}:0.5)$);
  		
  		\node[fill=white,circle,opacity=0.6,inner sep=2.7pt] (x2label) at ($(x2)+(135:0.2)$) {};
  		\node at ($(x2)+(135:0.2)$) {$v$};
  		\node (Label) at ($(c2)+(45:1.5)$) {$H_{\Compl{T}}$};
  		\end{tikzpicture}};
  	\draw[thick,->,-latex] (M) to (HS);
  	\draw[thick,->,-latex] (M) to (HSQ);
  	\draw[thick,->,-latex] (M) to (HT);
  	\draw[thick,->,-latex] (M) to (HTQ);
  \end{tikzpicture}
  \caption{An example for the non-uniqueness of tight cut decompositions in non-\uniformable hypergraphs}
  \label{fig:nonunique}
\end{figure}
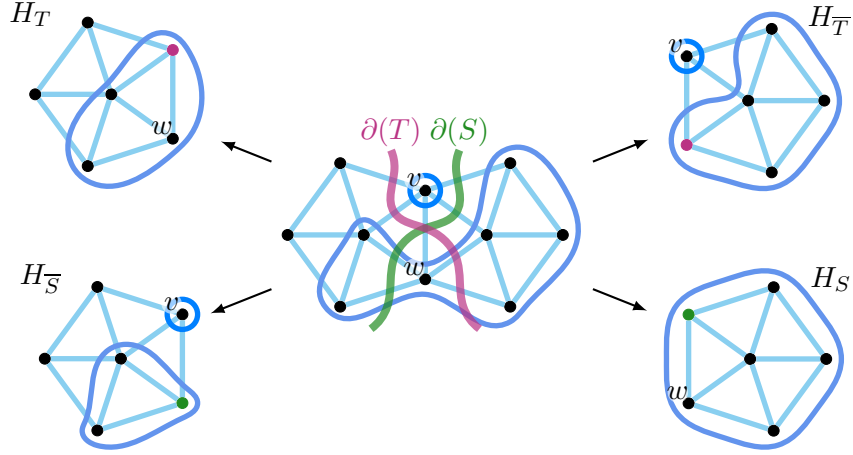
  
  This important result does not hold for general hypergraphs.
  A counterexample is depicted in \cref{fig:nonunique}.
  In its center it shows a matching covered hypergraph $H$ and two tight cuts $\Cut{H}{S}$, $\Cut{H}{T}$.
  The two hypergraphs obtained by contracting $S$ and $\Compl{S}$ are depicted below $H$ and the ones obtained by contracting $T$ and $\Compl{T}$ above $H$.
  The resulting hypergraphs $H_S$, $H_{\Compl{S}}$, $H_T$, $H_{\Compl{T}}$ have only trivial tight cuts, and they are not equivalent in the sense described above.
  Thus, $H_S,H_{\Compl{S}}$ and $H_T, H_{\Compl{T}}$ are two non-equivalent tight cut decompositions of $H$.
  
  So in general we cannot hope for a theorem showing uniqueness of the tight cut decomposition for hypergraphs.
  However, the hypergraph in \cref{fig:nonunique} is not \uniformable.
  What follows next is a series of lemmas analysing the properties of tight cuts and their contractions on \uniformable hypergraphs.
  Eventually this analysis will show that no such counterexample can be \uniformable.
  In a first step of our analysis, we show that the tight cuts of a tight cut contraction correspond to tight cuts in the original hypergraph. 
  
  \begin{lemma}
   	\label{lemma:tightcut}
   	Let $H $ be a matching covered hypergraph with a non-trivial tight cut $\Cut{H}{S}$ defined by $S \subseteq \Fkt{V}{H}$.
   	For $T \subseteq S$ we have that $\Cut{H}{T}$ is a tight cut in $H$ if and only if $\Cut{H_{\Compl{S}}}{T}$ is a tight cut in $H_{\Compl{S}}$.
   	Similarly, $T\subseteq \Compl{S}$ defines a tight cut in $H$ if and only if it defines a tight cut in $H_S$.
   \end{lemma}

   \begin{proof}
  	 We only prove the claim for $T\subseteq S$ since the other case is symmetric by interchanging $S$ and $\Compl{S}$.
  	 
  	 First, we show that every set $T\subseteq S$ that defines a tight cut in $H$ also defines a tight cut in $H_{\Compl{S}}$.
   	The hypergraph $H_{\Compl{S}}$ contains all $e \in \Fkt{E}{H}$ with $e \subseteq S$ and for every $e\in \Cut{H}{S}$ it contains the edge $e_{\Compl{s}} = \Brace{e \cap S} \cup \Set{\Compl{s}}$ where $\Compl{s}$ is a new vertex representing $\Compl{S}$.
   	Then 
   	\begin{equation*}
   		\Cut{H_{\Compl{S}}}{T} = \CondSet{e \in \Fkt{E}{H}}{e \subseteq S, e \cap T \neq \emptyset, e \setminus T \neq 
   		\emptyset} \cup \CondSet{e_{\Compl{s}}}{e \in \Cut{H}{S}, e \cap T \neq \emptyset}.
   	\end{equation*}
   	Let $M_1$ be a perfect matching of $H_{\Compl{S}}$.
   	There exists a unique edge $f \in M_1\setminus  \Fkt{E}{H}$, and this edge is of the form $f = m_{\Compl{s}}$ for some $m\in \Cut{H}{S}$.
   	Furthermore, there exists a perfect matching $M$ containing $m$ such that $M_1=\CondSet{e\in M}{e\subseteq S}\cup \Set{f}$.
   	As $\Cut{H}{T}$ is a tight cut, there exists a unique edge $m' \in \Cut{H}{T}\cap M$.
   	This edge has a non-empty intersection with $S$ because $T\subseteq S$ and $T\cap m'\neq \emptyset$.
   	Thus, $m'$ is either a subset of $S$ or it lies in the cut $\Cut{H}{S}$.
   	
   	In the first case, $m'$ is an edge of $H_{\Compl{S}}$ with $m' \in \Cut{H_{\Compl{S}}}{T} \cap M_1$, and in the latter $m'_{\Compl{s}} = m_{\Compl{s}} \in \Cut{H_{\Compl{S}}}{T} \cap M_1$.
   	Thus, $\Abs{\Cut{H_{\Compl{S}}}{T}\cap M_1}\geq \Abs{\Cut{H}{T} \cap M}$ holds.
   	On the other hand, if $e \in \Cut{H_{\Compl{S}}}{T} \cap M_1$, then either $e \subseteq S$ and $e \in \Cut{H}{T} \cap M$ or $e = m'_{\Compl{s}}$ for some $m' \in \Cut{H}{T} \cap M$.
   	In total we get $\Abs{\Cut{H_{\Compl{S}}}{T} \cap M_1} = \Abs{\Cut{H}{T} \cap M} = 1$, and thus $T$ also defines a tight cut in $H_{\Compl{S}}$.
   	
  	Second, suppose $T \subseteq S$ defines a tight cut in $H_{\Compl{S}}$.
  	We claim that $T$ also defines a tight cut in $H$.
  	Therefore, we have to show that $\Abs{M\cap \Cut{H}{T}}=1$ for all perfect matchings $M$ of $H$.
  	Given a perfect matching $M$ of $H$, let $M_1$ be the unique perfect matching of $H_{\Compl{S}}$ such that $M_1 = \CondSet{e \in M}{e \subseteq S} \cup \Set{m_{\Compl{s}}}$ where $m$ is the unique edge in $\Cut{H}{S}\cap M$.
  	Using similar arguments as above, it follows that $\Abs{\Cut{H}{S}\cap M} = \Abs{\Cut{H_{\Compl{S}}}{T}\cap M_1} = 1$.
  	Thus, $\Cut{H}{T}$ is a tight cut.
  \end{proof}

  We say that two sets $S,T$ \emph{cross} if all four of the following sets $S\cap T, S\cap \Compl{T}, \Compl{S} \cap T, \Compl{S} \cap \Compl{T}$ are non-empty, otherwise $S,T$ are called \emph{laminar}.
  It is sometimes more convenient to talk about the cuts $\Cut{}{S}$ and $\Cut{}{T}$.
  The cuts are called \emph{crossing}, if $S$ and $T$ cross, and otherwise they are called \emph{laminar}.
  
  With the help of the previous lemma we can show that every tight cut decomposition of a hypergraph corresponds to a maximal family of pairwise laminar non-trivial tight cuts.
  
  \begin{corollary}\label{cor:decompandlaminarfamily}
  Let $H$ be a matching covered hypergraph.
  Every tight cut decomposition $\Set{H_1,\ldots, H_i}$ of $H$ corresponds to a maximal family 
  \begin{equation*}
  	\mathcal F\subseteq \CondSet{S\subseteq \Fkt{V}{H}}{\Cut{H}{S}\text{ is a non-trivial tight cut.}}
  \end{equation*}
   such that $S,T\in \mathcal F$ implies that $S$ and $T$ are laminar.
  \end{corollary}

  \begin{proof}
  We prove this by induction on the number of hyperbricks in a tight cut decomposition $\Set{H_1,\ldots, H_i}$.
  If $i\in\Set{1,2}$, the assertion is true as $i=1$ implies that $H$ itself is a hyperbrick and thus $\mathcal{F}$ is empty.
  The case $i=2$ implies that $H_1$ and $H_2$ are both hyperbricks, obtained by the tight cut contractions of a single non-trivial tight cut $\Cut{}{S}$ and thus, by \cref{lemma:tightcut}, every tight cut in $H$ crosses $\Cut{}{S}$.
  
  Suppose that $i\geq 3$.
  Let $\Set{H_1,\ldots,H_i}$ be ordered such that there exists a tight cut $\Cut{}{S}$ where $H_i=H_{\Compl{S}}$ and $\Set{H_1,\ldots,H_{i-1}}$ is a tight cut decomposition of $H_S$.
  To see the existence consider the tree of tight cut contractions constructed by the tight cut decomposition procedure and use \cref{lemma:tightcut} to reorder the cuts chosen during its application in the desired way.
  Now, by induction, $\Set{H_1,\ldots,H_{i-1}}$ corresponds to a maximal family $\mathcal{F}'$ of pairwise laminar tight cuts in $H_S$.
  Let $\mathcal{F} \coloneqq \mathcal{F}' \cup \Set{S}$ and since $H_{\Compl{S}}$ is a hyperbrick by assumption, \cref{lemma:tightcut} yields the maximality of $\mathcal{F}$.
  \end{proof}
  
An important property of tight cuts in graphs is that crossing cuts can be uncrossed.
This means that if $\Cut{}{S}, \Cut{}{T}$ are two crossing tight cuts such that $S\cap T$ is of odd size, then also $\Cut{}{S\cap T}$ and $\Cut{}{S\cup T}$ are tight.
This uncrossing is not possible in general hypergraphs.
However, if we restrict our attention to \uniformable ones, similar results hold.

\begin{figure}[!h]
	\centering
	\begin{tikzpicture}[scale=0.65]
		\node (B) at (0,0) {};
		\def\OuterCurve{%
			($(B)+(90:3)$) .. ($(B)+(180:3.5)$) ..
			($(B)+(270:3)$) .. ($(B)+(360:3.5)$) .. ($(B)+(90:3)$)
		}
		\def\GraphCurve{%
			($(B)+(90:2)$) .. ($(B)+(135:2.5)$) .. ($(B)+(180:3)$) ..
			($(B)+(225:2.7)$) .. ($(B)+(270:1.8)$) .. ($(B)+(315:2.4)$) ..
			($(B)+(360:2.4)$) .. ($(B)+(45:2.7)$)
		}
		
		\coordinate (Upper) at ($(B)+(90:3)$);
		\coordinate (Lower) at ($(B)+(270:3)$);
		\coordinate (Left) at ($(B)+(180:3.5)$);
		\coordinate (Right) at ($(B)+(360:3.5)$);
		\def\CutCurveS{%
			(Upper) to[out=-40, in=100] ($(B)+(87:1)$) to[out=280, in=110] (Lower)
		}
		\def\CutCurveT{%
			(Left) to[out=30, in=190] ($(B)+(176:1)$) to[out=10, in=193] (Right)
		}
		\def\AreaSC{%
			\CutCurveS
      			-- (current bounding box.south east)
      			-- (current bounding box.north east)
      			-- cycle
      			}
      	\def\AreaS{%
			\CutCurveS
      			-- (current bounding box.south west)
      			-- (current bounding box.north west)
      			-- cycle
      			}
      	\def\AreaT{%
			\CutCurveT
      			-- (current bounding box.south east)
      			-- (current bounding box.south west)
      			-- cycle
      			}
      	\def\AreaTC{%
			\CutCurveT
      			-- (current bounding box.north east)
      			-- (current bounding box.north west)
      			-- cycle
      			}
		
		\begin{scope}
    		\clip[use Hobby shortcut, closed=true] \GraphCurve;
    		\begin{scope}
      			\clip \AreaS;
      			\clip \AreaT;
      			\fill[myLightBlue!60,opacity=0.5] (current bounding box.south west) rectangle (current bounding box.north east);
    		\end{scope}
    		\begin{scope}
      			\clip \AreaSC;
      			\clip \AreaTC;
      			\fill[myLightBlue!60,opacity=0.5] (current bounding box.south west) rectangle (current bounding box.north east);
    		\end{scope}
    		\begin{scope}
      			\clip \AreaSC;
      			\clip \AreaT;
      			\fill[myOrange!50,opacity=0.5] (current bounding box.south west) rectangle (current bounding box.north east);
    		\end{scope}
    		\begin{scope}
      			\clip \AreaS;
      			\clip \AreaTC;
      			\fill[myOrange!50,opacity=0.5] (current bounding box.south west) rectangle (current bounding box.north east);
    		\end{scope}
  		\end{scope}

		\draw[line width=1pt,gray,use Hobby shortcut,closed=true] \GraphCurve;
		\draw[myGreen,line width=2pt,use Hobby shortcut, closed=false] \CutCurveS;
		
		\draw[myViolet,line width=2pt,use Hobby shortcut, closed=false] \CutCurveT;
		\node[myGreen] (S) at ($(Upper)+(-1,0)$) {$S$};
		\node[myGreen] (SC) at ($(Upper)+(1,0)$) {$\Compl{S}$};
		\node[myViolet] (T) at ($(Right)+(0,-1)$) {$T$};
		\node[myViolet] (TC) at ($(Right)+(0,1)$) {$\Compl{T}$};
		
		\node (ST) at ($(B)+(220:1.5)$) {$S \cap T$};
		\node (STC) at ($(B)+(140:1.5)$) {$S \cap \Compl{T}$};
		\node (SCT) at ($(B)+(-40:1.5)$) {$\Compl{S} \cap T$};
		\node (SCTC) at ($(B)+(40:1.5)$) {$\Compl{S} \cap \Compl{T}$};

		\end{tikzpicture}
		\caption{Two crossing tight cuts $\Cut{}{S}$ and $\Cut{}{T}$.}
		\label{fig:crossing}
\end{figure}
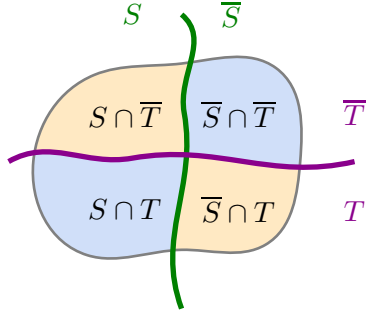

\begin{lemma}\label{lem:uncrossing}
Let $H$ be a matching covered, \uniformable hypergraph, $S,T\subseteq \Fkt{V}{H}$ be crossing sets such that $\Cut{H}{S}$ and $\Cut{H}{T}$ are tight.
The cut $\Cut{H}{S\cap T}$ is tight if and only if $\Cut{H}{S\cup T}$ is tight.
\end{lemma}

\begin{proof}
	Suppose that $\Cut{H}{S\cap T}$ is tight but $\Cut{H}{S\cup T}$ is not tight.
	Furthermore, let $\Multi{H}{m}$ be a multiplication of $H$ that is $r$-uniform for some $r\in \Z$.
	
	Let $M'$ be a perfect matching with $\Abs{M'\cap \Cut{H}{S\cup T}}\neq 1$.
	Assume $M' \cap \Cut{H}{S\cup T}=\emptyset$, then $\CondSet{e\subseteq S\cup T}{e\in M'}$ is a perfect matching of $\InducedSubgraph{H}{S\cup T}$, thus $\sum_{v\in S\cup T}\Fkt{m}{v}\equiv_r 0$.
	By \cref{obs:tightcut} there exists a perfect matching $M$ with $\Abs{M\cap \Cut{H}{S\cup T}}\geq 2$.
	If  $M' \cap \Cut{H}{S\cup T}\neq \emptyset$, then $\Abs{M' \cap \Cut{H}{S\cup T}} \geq 2$.
	So in any case there exists a perfect matching $M$ that has at least two edges in the cut $\Cut{H}{S\cup T}$.
	
	
	Every $e\in \Cut{H}{S\cup T}$ lies either in exactly one of the cuts $\Cut{H}{S}$ and $\Cut{H}{T}$ or in both.
	Thus, $\Abs{M\cap \Cut{H}{S\cup T}}=2$ and if $e_1, e_2$ are the two edges in the intersection of $M$ with $\Cut{H}{S\cup T}$, we can assume $e_1\in \Cut{H}{S}\setminus \Cut{H}{T}$ and $e_2\in \Cut{H}{T}\setminus \Cut{H}{S}$.
	It follows that $e_1\cap S\neq \emptyset$, $e_1\cap \Brace{\Compl{S} \cap \Compl{T}} \neq \emptyset$, $e_1\subseteq T$ or $e_1\subseteq \Compl{T}$, thus $e_1\subseteq \Compl{T}$ and $e_1\notin \Cut{H}{S\cap T}$.
	Similar, for $e_2$ we get $e_2\cap T\neq \emptyset$, $e_2\cap \Brace{\Compl{S} \cap \Compl{T}} \neq \emptyset$, $e_2\subseteq \Compl{S}$, and thus $e_2\notin \Cut{H}{S\cap T}$.
	This is a contradiction to the tightness of $\Cut{H}{S\cap T}$ because $M \cap \Cut{H}{S\cap T}\subseteq \Brace{M \cap \Cut{H}{S}} \cup \Brace{M \cap \Cut{H}{T}} = \Set{e_1,e_2}$, and therefore $M \cap \Cut{H}{S\cap T} = \emptyset$.
	
	The other direction follows by replacing $S$ and $T$ with $\Compl S$ and $\Compl T$ as $\Cut{H}{\Compl{S}} = \Cut{H}{S}$, $\Cut{H}{\Compl{T}}=\Cut{H}{T}$, $\Cut{H}{\Compl{S}\cap \Compl{T}}=\Cut{H}{S\cup T}$, and $\Cut{H}{\Compl{S} \cup \Compl{T}} = \Cut{H}{S\cap T}$.
\end{proof}

In \cref{fig:crossing} we see two crossing tight cuts and the two corners $S\cap T$ and $S\cup T$ on the diagonal marked in the same colour.
If we apply this lemma to $S$ and $\Compl{T}$ and observe that $\overline{S\cup \Compl{T}} = \Compl{S}\cap T$, we immediately obtain the following corollary.

\begin{corollary}
Let $H$ be a matching covered, \uniformable hypergraph and $S,T\subseteq \Fkt{V}{H}$ crossing sets such that $\Cut{H}{S}$ and $\Cut{H}{T}$ are tight.
Then the cut $\Cut{H}{S\cap \Compl{T}}$ is tight if and only if $\Cut{H}{\Compl{S}\cap T}$ is tight.
\end{corollary}

So recalling \cref{fig:crossing},
if both $\Cut{}{S}$ and $\Cut{}{T}$ are tight, then the two corners marked in dark colour, or the two corners marked in light colour are also tight cuts.
This is proven in the next lemma.

\begin{lemma}
Let $H$ be a matching covered, \uniformable hypergraph and $S,T\subseteq \Fkt{V}{H}$ be crossing sets such that $\Cut{H}{S}$ and $\Cut{H}{T}$ are tight.
If $\Cut{H}{S\cap T}$ is not tight, then $\Cut{H}{S\cap \Compl{T}}$ is tight.
\end{lemma}

\begin{proof}
If $\Cut{H}{S\cap T}$ is not tight, then also $\Cut{H}{S\cup T}$ is not tight.
Using the same arguments as in the proof of \cref{lem:uncrossing} we can find a perfect matching $M$ with $\Abs{M\cap \Cut{H}{S\cup T}}=2$.
We have seen that such a matching does not intersect $\Cut{H}{S\cap T}$.
This implies that $\sum_{v\in S\cap T}\Fkt{m}{v}\equiv_r 0$ where $\Multi{H}{m}$ is a $r$-uniform multiplication of $H$. 

Suppose that $\Cut{H}{S\cap \Compl{T}}$ is not tight.
By a symmetric argument, we can find a perfect matching $M$ with $\Abs{M\cap \Cut{H}{S \cup  \Compl{T}}}=2$ and $\Abs{M\cap \Cut{H}{S\cap \Compl{T}}}=0$ which implies that $\sum_{v\in S\cap \Compl{T}}\Fkt{m}{v}\equiv_r 0$.

Now, $\sum_{v\in S}\Fkt{m}{v} = \sum_{v\in S\cap \Compl{T}}\Fkt{m}{v}+ \sum_{v\in S\cap T}\Fkt{m}{v}\equiv_r 0$ which contradicts \cref{obs:tightcut}.
\end{proof}

We sum up the previous results in the following corollary.

\begin{corollary}\label{cor:uncrossing2}
If $H$ is a matching covered, \uniformable hypergraph, and $S,T\subseteq \Fkt{V}{H}$ are crossing sets such that $\Cut{H}{S}$ and $\Cut{H}{T}$ are tight, then $\Cut{H}{S\cap T}$ and $\Cut{H}{S\cup T}$, or  $\Cut{H}{S\cap \Compl{T}}$ and $\Cut{H}{\Compl{S}\cap T}$ are tight.
\end{corollary}

For two crossing tight cuts $\Cut{H}{S}, \Cut{H}{T}$ we can always assume that $\Cut{H}{S\cap T}$ and $\Cut{H}{S\cup T}$ are tight after possibly replacing $T$ by $\Compl{T}$. 
In the graph case, all four sets $S,T, S\cap T$, and $S\cup T$ would be of odd size.
In \uniformable hypergraphs it is possible that $\Cut{H}{S}$ and $\Cut{H}{T}$ coincide even if $S\neq T$ and $S\neq \Compl{T}$. If two crossing sets define distinct tight cuts, then similar parity results as in the graph case hold.
Therefore, we need the following observation concerning tight cuts.
 \begin{observation}\label{obs:subset}
 If $S$ defines a tight cut in a \uniformable hypergraph $H$ and $\Cut{H}{A}$ is a non-empty cut with $\Cut{H}{A}\subseteq \Cut{H}{S}$, then  $\Cut{H}{A} = \Cut{H}{S}$.
 \end{observation} 
 \begin{proof}
 Suppose there exists a hyperedge $e^*\in  \Cut{H}{S}\setminus \Cut{H}{A}$. Let $M$ be a perfect matching of $H$ containing $e^*$. As $|M\cap \Cut{H}{S}|=1$, it follows that $M\cap \Cut{H}{A}=\emptyset$. This implies that $m(A)\equiv 0$ mod $r$ for any function $m:V(H)\rightarrow \mathbb N$ and $r\in \mathbb N$ such that $H^{(m)}$ is $r$-uniform. By \cref{obs:tightcut}, there exists a perfect matching $M'$ with $|M'\cap \Cut{H}{A}|\geq 2$. Thus, $|M'\cap \Cut{H}{S}|\geq 2$, contradicting that $\Cut{H}{S}$ is tight.
 \end{proof}
 
\begin{lemma}\label{obs:parity}
Let $H$ be a matching covered, \uniformable hypergraph, $S, T\subseteq \Fkt{V}{H}$ two crossing sets such that $S,T$ and $S\cup T$ define tight cuts in $H$.
If $\Cut{H}{S} \neq \Cut{H}{T}$, then for any function $m : \Fkt{V}{H}\rightarrow \Z_{\geq 1}$ with the property that $\Multi{H}{m}$ is $r$-uniform for some $r\in \Z$ we have $\Fkt{m}{S}\equiv_r \Fkt{m}{S\cap T}\equiv_r \Fkt{m}{S\cup T}\equiv_r \Fkt{m}{T}$.
\end{lemma}

\begin{proof}
  We first show that $\Cut{H}{S}\neq \Cut{H}{T}$ implies $\Cut{H}{S}\neq \Cut{H}{S\cap T}$ and $\Cut{H}{S}\neq \Cut{H}{S\cup T}$.
  Suppose that $\Cut{H}{S} = \Cut{H}{S\cap T}$.
  As $S$ and $T$ are crossing, the set $S\cap \Compl{T}$ is not empty and because $H$ is connected $\Cut{H}{S\cap \Compl{T}}\neq \emptyset$.
  Suppose that there exists an edge $e\in \Cut{H}{S\cap \Compl{T}}\setminus \Cut{H}{S}$, then $e\subseteq S$ or $e\subseteq \Compl{S}$.
  Because of $e\in\Cut{H}{S\cap \Compl{T}}$ we have $e\subseteq S$ and $e \cap \Brace{S\cap T} \neq \emptyset$.
  But then $e\in \Cut{H}{S\cap T}\setminus \Cut{H}{S}$ which is not possible as we assumed $\Cut{H}{S\cap T}=\Cut{H}{S}$.
  Thus, $\Cut{H}{S\cap \Compl{T}}\subseteq \Cut{H}{S}$, and by \cref{obs:subset} equality holds.
  Now, $e'\in \Cut{H}{S}$ implies $e'\in \Cut{H}{S\cap T}=\Cut{H}{S\cap \Compl{T}}$, and thus $e'\cap \Brace{S\cap T} \neq \emptyset$ and $e'\cap \Brace{S\cap \Compl{T}} \neq \emptyset$, in particular, we get $e'\in \Cut{H}{T}$.
  It follows that $\Cut{H}{S}= \Cut{H}{T}$; a contradiction.
  
  If we interchange $S$ with $\Compl{S}$ and $T$ with $\Compl{T}$ the argument above shows that $\Cut{H}{S} = \Cut{H}{\Compl{S}}\neq \Cut{H}{\Compl{S}\cap \Compl{T}}=\Cut{H}{S\cup T}$.
  
  Now, let $e_1\in \Cut{H}{S}\setminus \Cut{H}{S\cap T}$, and choose a perfect matching $M$ containing $e_1$.
  Let $e_2\in M\cap \Cut{H}{S\cap T}$.
  It follows that $e_2\in \Cut{H}{T}$, $e_2\notin \Cut{H}{S}$, $e_2\notin \Cut{H}{S\cup T}$, $e_1\notin \Cut{H}{T}$, and $e_1\in \Cut{H}{S\cup T}$.
  Thus, $e_1\cap \Brace{\Compl{S} \cap \Compl{T}} \neq \emptyset$ and $e_1\subseteq \Compl{T}$, which implies 
  \begin{equation*}
  	\sum_{v\in S}\Fkt{m}{v}\equiv_r \sum_{v\in e_1\cap S}\Fkt{m}{v}= \sum_{v\in e_1\cap \Brace{S \cup T}}\Fkt{m}{v}\equiv_r \sum_{v\in S\cup T}\Fkt{m}{v}.
  \end{equation*}
  For $e_2$ we get $e_2\cap \Brace{S\cap T} \neq \emptyset$ and $e_2\subseteq S$, thus 
  \begin{equation*}
  	\sum_{v\in T}\Fkt{m}{v}\equiv_r\sum_{v\in e_2\cap T}\Fkt{m}{v}=\sum_{v\in e_2\cap \Brace{S\cap T}}\Fkt{m}{v}\equiv_r \sum_{v\in S\cap T}\Fkt{m}{v}.
  \end{equation*}
  Starting the same argument with $e_1 \in \Cut{H}{S}\setminus \Cut{H}{S\cup T}$ gives $m(S)\equiv_r m(S\cap T)$ and $m(T)\equiv_r m(S\cup T)$.
\end{proof}

If $\Cut{H}{S}, \Cut{H}{T}$ are non-trivial tight cuts it is possible that $\Cut{H}{S\cap T}$ and $\Cut{H}{S\cup T}$ are trivial.
In this case, $\Abs{S\cap T}=1=\Abs{\Compl{S}\cap \Compl{T}}$.
The following lemma shows that in this degenerate case the same, up to parallel edges, tight cut contractions are obtained.

\begin{lemma}\label{lemma:separationcut}
 	Let $H$ be a matching covered, \uniformable hypergraph, and $S, T \subseteq \Fkt{V}{H}$ be sets such that $\Abs{S\cap T} = 1 = \Abs{\Compl{S}\cap \Compl{T}}$.
 	If $\Cut{H}{S}$ and $\Cut{H}{T}$ are non-trivial tight cuts with $\Cut{H}{S}\neq \Cut{H}{T}$, then the tight cut contractions w.r.t.\ $\Cut{H}{S}$ yield the same two hypergraphs as the contractions w.r.t.\ $\Cut{H}{T}$ up to parallel edges.
\end{lemma}

\begin{proof}
	As $H$ is \uniformable there exists a function $m:\Fkt{V}{H}\rightarrow \Z_{\geq 1}$ such that $\Multi{H}{m}$ is $r$-uniform for some $r\in \Z$.
	We denote the unique vertex in $S\cap T$ by $v^*$ and the unique vertex in $\Compl{S}\cap \Compl{T}$ by $w^*$.
	
	A crucial observation is that $\Cut{H}{S}\cap \Cut{H}{T}=\Set{\Set{v^*, w^*}}$ if $\{v^*, w^*\}\in \Fkt{E}{H}$, and $\Cut{H}{S}\cap \Cut{H}{T}=\emptyset$ otherwise.
	Namely, there exists an integer $k\in \{1,\ldots, r-1\}$ such that $\Fkt{m}{S}\equiv \Fkt{m}{T}\equiv \Fkt{m}{S\cap T}\equiv \Fkt{m}{S\cup T}\equiv_r k$ by \cref{obs:tightcut} and \cref{obs:parity}.
	This implies that $\Fkt{m}{v^*}=k$ and $\Fkt{m}{w^*}=r-k$.
	If there exists an edge $e\in \Cut{H}{S}\cap \Cut{H}{T}$, then $e$ contains $v^*$ and $w^*$.
	As $\sum_{v\in e}\Fkt{m}{v}=r$ and $\Fkt{m}{v}\geq 1$ for all $v\in \Fkt{V}{H}$, $e$ is equal to $\Set{v^*, w^*}$.

	Now, we show that $H_S$ and $H_{\Compl{T}}$ are equivalent up to multiple edges.
	We define a function $\phi : \Fkt{V}{H_{\Compl{T}}}\rightarrow \Fkt{V}{H_S}$ by 
	\begin{equation*}
		\Fkt{\phi}{v}\coloneqq\begin{cases}
			v &, \text{ if } v\in \Compl{S}\cap T\\
			s &, \text{ if } v = v^*\in S\cap T\\
			w^* &, \text{ if } v=\Compl{t}\in \Fkt{V}{H_{\Compl{T}}}\setminus T.
		\end{cases}
	\end{equation*}
	Observe that $\Fkt{V}{H_{\Compl{T}}} = \Brace{\Compl{S}\cap T}\cup \Set{v^*}\cup \Set{\Compl{t}}$ and $\Fkt{V}{H_S}=\Brace{\Compl{S}\cap T }\cup \Set{w^*}\cup \Set{s}$.
	Thus, $\phi$ is well defined and bijective.
	We claim that $\phi$ extends to a function from $\Fkt{E}{H_{\Compl{T}}}$ to $\Fkt{E}{H_S}$ via $\Fkt{\phi}{e}=\CondSet{\Fkt{\phi}{v}}{v\in e}$ for all $e\in \Fkt{E}{H_{\Compl{T}}}$.
	First, we show that $\Fkt{\phi}{e}\in \Fkt{E}{H_S}$ for all $e\in \Fkt{E}{H_{\Compl{T}}}$.
	\begin{itemize}
		\item  If $e\subseteq \Compl{S}\cap T$, then $\Fkt{\phi}{e}=e\in \Fkt{E}{H_S}$.
		
		\item  If $e\subseteq T$ and $v^*\in e$, then $\Fkt{\phi}{e}=e\cap \Compl{S}\cup \{s\}\in \Fkt{E}{H_S}$.
		
		\item  If $\Compl{t}\in e$, then there exists an edge $\tilde e\in \Fkt{E}{H}$ with $\Brace{\tilde e\cap T}\cup \Set{\Compl{t}}=e$.
		This edge $\tilde e$ lies also in $\Cut{H}{S\cap T}$, $\Cut{H}{S\cup T}$, or in both.
		Thus, $\tilde e\cap \Set{v^*, w^*}\neq \emptyset$.
		
		If $\tilde e=\Set{v^*, w^*}$, then $e=\Set{v^*,\Compl{t}}$, and $\Fkt{\phi}{e} = \Set{s,w^*}=\Brace{\tilde e\cap \Compl{S}}\cup \Set{s}\in \Fkt{E}{H_S}$.
		
		Next, we consider the case $\tilde e\cap \Set{v^*,w^*}=\Set{w^*}$.
		In this case, $\tilde e\subseteq \Compl{S}$ and $\tilde e\in \Cut{H}{T}\setminus \Cut{H}{S}$.
		We get $\Fkt{\phi}{e} = \Brace{\Brace{\tilde e\cap \Compl{S}}\cap T}\cup \Set{w^*} = \Brace{e \setminus \Set{\Compl{t}}}\cup \Set{w^*}=\tilde e\in \Fkt{E}{H_S}$.
		
		It remains to consider the case $\tilde{e} \cap \Set{v^*,w^*}=\Set{v^*}$.
		Now, $\tilde e\subseteq S$.
		Let $\tilde M$ be a perfect matching containing $\tilde e$ and let $e'\in \tilde M\cap \Cut{H}{S}$.
		It follows that $w^*\in e'$, and $e'\in \Cut{H}{S}\setminus \Cut{H}{T}$ thus $e'\subseteq \Compl{T}$.
		Furthermore, $\Brace{e'\cap \Compl{S}} \cup \Set{s}= \Set{w^*,s}$ implies $\Set{w^*,s}\in \Fkt{E}{H_S}$.
		Now, $\Fkt{\phi}{e} = \Brace{\tilde e\cap \Brace{\Compl{S}\cap T}}\cup \Set{s}\cup \Set{w^*}=\Set{s,w^*}\in \Fkt{E}{H_S}$.
	\end{itemize}
	
	Next, we show that $\phi:\Fkt{E}{H_{\Compl{T}}}\rightarrow \Fkt{E}{H_S}$ is surjective, i.e.\@ for every $e\in \Fkt{E}{H_S}$ there exists $f\in \Fkt{E}{H_{\Compl{T}}}$ such that $\Fkt{\phi}{f}=e$.
	\begin{itemize}
		\item If $e\subseteq \Compl{S}\cap T$, then $e\in \Fkt{E}{H_{\Compl{T}}}$ and $\Fkt{\phi}{e}=e$.
		
		\item If $e\subseteq \Compl{S}$ and $w^*\in e$, then $e\in \Cut{H}{T}$, thus $\Brace{e\cap T}\cup \Set{\Compl{t}}\in \Fkt{E}{H_{\Compl{T}}}$, and $\Fkt{\phi}{\Brace{e\cap T}\cup \Set{\Compl{t}}}= e$.
		
		\item If $s\in e, w^*\notin e$, then $e=\tilde e\cap \Compl{S}\cup\Set{s}$ where $\tilde e\in \Cut{H}{S}$ with $v^*\in \tilde e$ and $w^*\notin \tilde e$.
		It follows that $\tilde e\notin \Cut{H}{T}$ and as $v^*\in \tilde e$ this means $\tilde e\subseteq T$.
		Thus, $\tilde e\in \Fkt{E}{H_{\Compl{T}}}$ and $\Fkt{\phi}{\tilde e} = e$.
		
		\item If $s, w^*\in e$, then $e=\Brace{\tilde e\cap \Compl{S}} \cup\Set{s}$ where $\tilde e\in \Cut{H}{S}$ with $w^*\in \tilde e$.
		If also $v^*\in \tilde e$, then $\tilde e=\Set{v^*, w^*}$ and $\Set{v^*, \Compl{t}}\in \Fkt{E}{H_{\Compl{T}}}$ with $\Fkt{\phi}{\Set{v^*, \Compl{t}}}=\Set{s,w^*}=e$.
		In the case $v^*\notin \tilde e$, we know that $\tilde e\notin \Cut{H}{T}$ and as $w^*\in \tilde e$ we get $\tilde e\subseteq \Compl{T}$.
		This means that $e=\Set{s,w^*}$.
		Now, let $\tilde M$ be a perfect matching containing $\tilde e$ and $e' \in \tilde{M}\cap \Cut{H}{T}$.
		We get $e' \subseteq S$ because $v^*\in e'$ and $e' \notin \Cut{H}{S}$.
		Therefore, $\Set{v^*,\Compl{t}}=\Brace{e' \cap T} \cup \Set{\Compl{t}}\in \Fkt{E}{H_{\Compl{T}}}$ and $\Fkt{\phi}{\Set{v^*, t}} = \Set{s,w^*}=e$.
	\end{itemize}
	It remains to show that $\Fkt{\phi}{e} = \Fkt{\phi}{e'}$ implies that $e$ and $e'$ are parallel for all $e,e' \in \Fkt{E}{H_{\Compl{T}}}$.
	This is clear if $\Fkt{\phi}{e} \subseteq \Compl{S}\cap T$ as $\phi$ is the identity on $\Compl{S}\cap T$.
	If $\Fkt{\phi}{e} \subseteq \Compl{S}$ and $w^*\in \Fkt{\phi}{e}$, then $\Compl{t} \in e \cap e'$ and $e \cap T = e' \cap T$, i.e.\@ $e$ and $e'$ are parallel.
	Now, if $s\in \Fkt{\phi}{e}$ and $w^*\notin \Fkt{\phi}{e}$, then $v^*\in e\cap e'$, and $\Compl{t} \notin e$, $\Compl{t}\notin e'$.
	This implies that $e,e'\subseteq T$ and thus $e \setminus \Set{v^*} = e' \setminus \Set{v^*}$.
	In total, $e$ and $e'$ contain the same vertices and are therefore parallel.
	If both $s$ and $w^*$ lie in $\Fkt{\phi}{e}$, then $\Fkt{\phi}{e} = \Set{s,w^*}$ and $e = \Set{v^*,\Compl{t}}=e'$.
	
	In total, $H_{\Compl{T}}$ and $H_S$ are isomorphic via $\phi$ up to parallel edges.
	
	By similar arguments, we can show that $H_{T}$ and $H_{\overline S}$ are isomorphic up to parallel hyperedges using the function $\phi:\Fkt{V}{H_{T}}\rightarrow \Fkt{V}{H_{\overline S}}$ defined by
	 \begin{align*}
	 	\phi(v) \coloneqq \begin{cases}
	  v, & v\in S\cap \overline T\\
	  	\overline s, & v = w^*\in \overline S\cap \overline T\\
	  	v^*, & v=t\in \Fkt{V}{H_{T}}\setminus \overline T.
	  \end{cases}
	  \end{align*}
	  \par \vspace{-1.7\baselineskip}
		\qedhere
\end{proof}
 
With \cref{lemma:separationcut} we have the last piece in place to be able to prove our main result: the uniqueness of the tight cut decomposition in \uniformable matching covered hypergraphs.

Instead of considering the outcome of two different tight cut decomposition procedures, we will use \cref{cor:decompandlaminarfamily} to directly compare the choices of tight cuts made during the process by looking at maximal families of pairwise laminar tight cuts.
By \cref{lemma:tightcut} we can apply the tight cuts in one of those families in any order and will always obtain the same list of hyperbricks.
This allows us to quickly settle two major cases by applying induction.
The remaining cases are those in which for every choice of two tight cuts, one from each of the two families, the two cuts are either the same but differ in shores, or cross.
These cases need further handling and some applications of the tools for uncrossing tight cuts obtained earlier in this section.
 
\begin{theorem}
 	\label{thm:unique}
 	Any two tight cut decomposition procedures of a matching covered, \uniformable hypergraph yield the same list of indecomposable hypergraphs up to  parallel edges.
\end{theorem}
\begin{proof}
 	As in the graph case, we use induction on the number of vertices of $H$.
 	If $H$ has at most three vertices, then $H$ has no non-trivial tight cuts.
 	Now, suppose the theorem holds for all hypergraphs $H$ with $\Abs{\Fkt{V}{H}} \leq l$.
 	Let $H$ be a hypergraph on $\Brace{l+1}$ vertices and $\mathcal{F}, \mathcal{F}'$ two maximal families of pairwise laminar tight cuts.
 	If $H$ has no non-trivial tight cuts, then $\mathcal{F} = \mathcal{F}' = \emptyset$.
 	Otherwise, we distinguish the following cases where the first and second one are identical to the ones in the graph case.
 	
 	\begin{description}
 		\item[Case 1.] $\mathcal{F}$ and $\mathcal{F}'$  have a common member $S$.
 			We can start a tight cut decomposition procedure with $S$ resulting in the matching covered, \uniformable hypergraphs $H_S$ and $H_{\Compl{S}}$ with at most $l$ vertices.
 			By induction hypothesis, $\mathcal{F} \setminus S$ and $\mathcal{F}' \setminus S$ yield the same decompositions on $H_S$ and $H_{\Compl{S}}$.
 			Thus, the decomposition procedures associated to $\mathcal{F}$ and $\mathcal{F}'$ yield the 
 			same list of indecomposable hypergraphs.
 			
 		\item[Case 2.] There exist $S \in \mathcal{F}$, $T \in \mathcal{F}'$ such that $S$ and $T$ are laminar.
 			Let $\mathcal{F}''$ be any maximal family of pairwise laminar tight cuts containing both $S$ and $T$.
 			By the first case, every tight cut decomposition associated to $\mathcal{F}$ and $\mathcal{F}''$ 
 			as well as $\mathcal{F}'$ and $\mathcal{F}''$ yield the same list of indecomposable hypergraphs.

 		\item[Case 3.] There exist $S \in \mathcal{F}$, $T \in \mathcal{F}'$ such that $\Cut{H}{S}=\Cut{H}{T}$.
 		 		Then, $\Cut{H}{S \cap T}=\Cut{H}{S\cap \Compl{T}}=\Cut{H}{\Compl{S} \cap \Compl{T}}=\Cut{H}{\Compl{S} \cap T}=\Cut{H}{S}$.
 		 		If one of the four sets $S \cap T, S\cap \Compl{T}, \Compl{S}\cap T, \Compl{S}\cap \Compl{T}$ has size at least two, then it defines a non-trivial tight cut laminar to $S$ and $T$, and we can proceed as in the second case.
 		 		Otherwise, $\Abs{S}=\Abs{T}=2$, $\Abs{\Fkt{V}{H}}=4$, and $H$ consists of edges of the form $\Set{\Fkt{V}{H}}$.
 		 		In this case, each of the four hypergraphs $H_S, H_{\Compl{S}}, H_T, H_{\Compl{T}}$ has three vertices and some parallel edges containing all three vertices.
 		 		As a hypergraph with at most three vertices has only trivial cuts, we have $\mathcal F=\Set{S}$ and $\mathcal F'=\Set{T}$ and the tight cut contractions w.r.t.\@ $\Cut{H}{S}$ and $\Cut{H}{T}$ are isomorphic.

 		 \item[Case 4.] Suppose that neither of the previous cases holds, and choose $S \in \mathcal{F}$, $T \in \mathcal{F}'$ arbitrary.
 		 	
	 		By \cref{cor:uncrossing2}, we can assume that $S \cap T$ and $S \cup T$ define tight cuts.
 		 	If $\Abs{S \cap T} > 1$ or $\Abs{\Fkt{V}{H} \setminus \Brace{S\cup T}} > 1$, then $U \coloneqq S \cap T$ or $U \coloneqq S \cup T$ defines a non-trivial tight cut which is laminar to both $S$ and $T$.
 			Let $\mathcal F''$ be a maximal family of pairwise laminar tight cuts containing $U$.
 		 	By the second case, $\mathcal F$ and $\mathcal F''$ as well as $\mathcal F''$ and $\mathcal F'$ yield the same tight cut decomposition.
 		 	 		
 			In the remainder we assume that $\Abs{S \cap T} = 1$ and $\Abs{\Compl{S} \cap \Compl{T}} = 1$.
 			If $\mathcal{F}$ and $\mathcal{F}'$ consist of just one cut, then by \cref{lemma:separationcut} the tight cut contractions with respect to $\Cut{H}{S}$ and $\Cut{H}{T}$ yield isomorphic hypergraphs.
 			Otherwise, we may assume by symmetry that $\mathcal{F}$ contains another cut $\Cut{H}{S'}$.
 			As $S'$ is laminar to $S$, we can assume $S' \cap S = \emptyset$, i.e.\@ $S'\subseteq \Compl{S}$.
 			If $S'$ and $T$ are laminar or $\Cut{H}{S'}=\Cut{H}{T}$, then we are in one of the previous cases.
 				 	 	
 			It remains to consider the case that $S'$ and $T$ are crossing sets defining different cuts. 
 			We know that $S' \cap T \neq \emptyset$ and $S' \cap \Compl{T} \neq \emptyset$.
 			It follows that $S'\cap \Compl{T} = S' \cap \Compl{S} \cap \Compl{T} = \Compl{S}\cap \Compl{T}$ because $\Abs{\Compl{S}\cap \Compl{T}} = 1$.
 			Let $w^* \in \Compl{S}\cap \Compl{T}$.
 			We can write $S'$ as $S'=\Set{w^*}\cup \Brace{S'\cap T}$.
 			Suppose that $\Cut{H}{S'\cap T}$ is a tight cut, and let $m : \Fkt{V}{H} \to \Z_{\geq 1}$ be a function such that $\Multi{H}{m}$ is $r$-regular for some $r \in \Z$.
 			On the one hand, due to \cref{obs:parity}, $\Fkt{m}{S'} \equiv_r \Fkt{m}{S'\cap T} \equiv_r \Fkt{m}{T} \equiv_r \Fkt{m}{T \cap S}$.
 			On the other hand, $\Fkt{m}{S'} = \Fkt{m}{S'\cap T} + \Fkt{m}{S'\cap \Compl{T}} = \Fkt{m}{S'\cap T} + \Fkt{m}{\Compl{S}\cap \Compl{T}} \equiv_r \Fkt{m}{T} - \Fkt{m}{S\cap T} \equiv_r 0$.
 			However, as $\Cut{H}{S'}$ is a tight cut $\Fkt{m}{S'}$ cannot be divisible by $r$.
 			Thus, $\Cut{H}{S'\cap T}$ is not tight and therefore $\Cut{H}{\Compl{S}'\cap T}$ is a tight cut.
 			This cut is non-trivial because otherwise $S'\cap T = \overline S\cap T$ would follows, which implies that $S'=\{w^*\}\cup \left(S'\cap T\right) = \overline S$, contradicting the choice of $S'$. This means that $U=\overline S'\cap T$ defines a non-trivial tight cut that does not cross any cut of $\mathcal F$ or $\mathcal F'$. Again, by considering a maximal family $\mathcal F''$ of pairwise non-trivial tight cuts containing $U$ and using the second case, we get that $\mathcal F$ and $\mathcal F'$ yield equivalent tight cut decompositions. \qedhere
 	\end{description}
 \end{proof}

\section{The Perfect Matching Polytope and Tight Cuts}\label{sec:polytope}


One of the original motivations for the definition of tight cuts was to determine the dimension of the perfect matching polytope of a graph.
Edmonds gave the following description of said polytope.
\begin{theorem}[Edmonds, 1965 \cite{edmonds1965maximum}]\label{thm:edmondsperfmat}
	The perfect matching polytope of a graph $G$ is determined by the following set of inequalities:
	\begin{enumerate}
		\renewcommand{\labelenumi}{\textbf{\theenumi}}
		\renewcommand{\theenumi}{(C\arabic{enumi})}
		\item \label{constr:nonneg} $\Fkt{x}{e}\geq 0 ~\text{for all}~e\in\Fkt{E}{G}$,
		\item \label{constr:degree} $\Fkt{x}{\Cut{}{u}} = 1~\text{for all}~u\in\Fkt{V}{G}$,
		\item \label{constr:oddset} $\Fkt{x}{\Cut{}{U}}\geq 1~\text{for all}~U\subseteq\Fkt{V}{G}~\text{with}~\Abs{U}\geq 
	3~\text{odd}$.
	\end{enumerate}
\end{theorem}
We call the inequalities from \cref{constr:degree} \emph{degree constraints} and the inequalities from \cref{constr:oddset} \emph{odd set constraints}.
A non-trivial tight cut for a graph $G$ is the special case of an odd set constraint.
Edmonds, Pulleyblank, and Lov{\'a}sz showed that the dimension of the perfect matching polytope of a graph is determined by its number of vertices, edges, and the number of bricks in a tight cut decomposition.
Furthermore, one can use the tight cut decomposition to give a non-redundant description of the perfect matching polytope~\cite{edmonds1982}.

An important observation is that a tight cut induces a decomposition of the perfect matching polytope.
We establish a similar link between the perfect matching polytope of a hypergraph and the perfect matching polytopes of its tight cut contractions.

Bipartite graphs play a special role in the theory of matching covered graphs.
In many cases they are more accessible.
The fact that, usually, one distinguishes between bricks, non-bipartite graphs without non-trivial tight cuts, and braces, bipartite graphs without non-trivial tight cuts, can be understood in the context of tight cut contractions.
To be more precise: Any tight cut contraction of a bipartite matching covered graph is again bipartite.
Hence the tight cut decomposition of a bipartite graph will never contain a brick and therefore, in many cases it suffices to describe exactly the braces to make statements about properties of bipartite matching covered graphs.

There are several possible ways to generalise bipartite graphs to hypergraphs~\cite{bergebook}.
One of them is the notion of a balanced hypergraph as follows.
A \emph{cycle} in a hypergraph $H$ is a sequence $C=x_1e_1x_2\dots x_te_tx_1$ with $x_i\neq x_j$ and $e_i\neq e_j$ for all $i\neq j$ and $x_i,x_{i+1}\in e_i$ for all $i\in\Set{1,\dots t}$ with $x_{t+1}=x_1$.
The number of edges $t$ in $C$ is the \emph{length} of $C$.
A cycle $C$ is \emph{strong}, if $x_j\in e_i$ implies $j\in\Set{i,i+1}$ with $x_{t+1}=x_1$.

\begin{definition}[Balanced Hypergraph]
	A hypergraph is called \emph{balanced} if it does not contain a strong cycle of odd length.
\end{definition}

There are several characterisations of balanced hypergraphs, some of these can be found in Berge's book on hypergraphs.
We will add a new characterisation of balanced hypergraphs in terms of their perfect matching polytope.
This characterisation together with the decomposition theory for the perfect matching polytope of hypergraphs obtained from the tight cut decomposition will yield the main result of this section:

Tight cut contractions of balanced \uniformable hypergraphs are again balanced.

We start with a formal definition of the perfect matching polytope of a hypergraph.
First, we define the incidence vector of a set, which is the usual way to assign a vector to a set.

\begin{definition}
Given some ground set $N$, the incidence vector of a subset $S\subseteq N$ is the vector $\chi^S\in \mathds{Q}^N$ defined by
\[
\chi^S_i\coloneqq \begin{cases}
1 & \text{ if }i\in S\\
0 & \text{ else.}
\end{cases}
\]
\end{definition}

Now, we can define the perfect matching polytope as the convex hull of the incidence vectors of all perfect matchings.

\begin{definition}
Let $H$ be a hypergraph and $\mathcal{M}$ its set of perfect matchings.
The perfect matching polytope of $H$ is denoted by $\MatPol{H}$ and defined as 
\[
	\MatPol{H} \coloneqq \CondSet{\sum_{M \in \mathcal{M}} \lambda_M \chi^M}{\sum_{M \in \mathcal{M}} \lambda_M = 1, \lambda_M \geq 0~\text{for all}~ M \in \mathcal{M}}.
\]
\end{definition}

We have seen that if $\Cut{H}{S}$ is a tight cut in a hypergraph $H$, then every perfect matching in $H$ corresponds to a pair $M_S, M_{\Compl{S}}$ of perfect matchings in $H_S$ and $H_{\Compl{S}}$ that agree on $\Cut{H}{S}$.
We say that $M_S$ and $M_{\Compl{S}}$ agree on $\Cut{H}{S}$ if there exists an edge $e\in \Cut{H}{S}$ such that $e_s\in \Cut{H_S}{s}\cap M_S$ and $e_{\Compl{s}}\in \Cut{H_{\Compl{S}}}{\Compl{s}}\cap M_{\Compl{S}}$.
This decomposition carries over to the corresponding perfect matching polytopes.
 
First, we show how to decompose a vector in $\mathds{Q}^{\Fkt{E}{H}}$ into vectors in $\mathds{Q}^{\Fkt{E}{H_S}}$ and $\mathds{Q}^{\Fkt{E}{H_{\Compl{S}}}}$.
Every $e\in \Cut{H}{S}$ corresponds to an edge $e_s\in \Fkt{E}{H_S}$ and an edge $e_{\Compl{s}}\in \Fkt{E}{H_{\Compl{S}}}$.
Thus, every $x\in \mathds{Q}^{\Fkt{E}{H}}$ corresponds to $x^S\in \mathds{Q}^{\Fkt{E}{H_S}}$ and $x^{\Compl{S}}\in \mathds{Q}^{\Fkt{E}{H_{\Compl{S}}}}$ via $\Fkt{x^S}{e_s} \coloneqq \Fkt{x^{\Compl{S}}}{e_{\Compl{s}}} \coloneqq  \Fkt{x}{e} $ for all $e \in \Cut{H}{S}$ and $\Fkt{x^S}{e} \coloneqq \Fkt{x}{e} $, $\Fkt{x^{\Compl{S}}}{e'} \coloneqq \Fkt{x}{e'}$ for $e \subseteq \Compl{S}$ and $e' \subseteq S$.
On the other hand, if $x^S \in \mathds{Q}^{\Fkt{E}{H_S}}$ and $x^{\Compl{S}} \in \mathds{Q}^{\Fkt{E}{H_{\Compl{S}}}}$ are given with $\Fkt{x^S}{e_s} = \Fkt{x^{\Compl{S}}}{e_{\Compl{s}}}$ for all $e\in \Cut{H}{S}$, then we say that \emph{$x^S$ and $x^{\Compl{S}}$ agree on $\Cut{H}{S}$}.
In this case, we define a vector $x^S\oplus x^{\Compl{S}}\in \mathds{Q}^{\Fkt{E}{H}}$ by
\begin{align*}
	 \Fkt{x}{e} \coloneqq \begin{cases}
	\Fkt{x^S}{e} & \text{if }e\subseteq \Compl{S}\\
	\Fkt{x^{\Compl{S}}}{e} & \text{if }e\subseteq S\\
	\Fkt{x^S}{e_s} & \text{if }e\in \Cut{H}{S}.
	\end{cases}
\end{align*} 
This implies that there is a one-to-one correspondence between vectors $x\in \mathds{Q}^{\Fkt{E}{H}}$ and pairs of vectors $x^S\in \mathds{Q}^{\Fkt{E}{H_S}}$, $x^{\Compl{S}}\in \mathds{Q}^{\Fkt{E}{H_{\Compl{S}}}}$ agreeing on $\Cut{H}{S}$.
For vectors we use the same shorthand as for multiplicity that is if $x\in \mathds{Q}^{\Fkt{E}{H}}$ and $E' \subseteq \Fkt{E}{H}$, then $\Fkt{x}{E'} \coloneqq \sum_{e \in E'} \Fkt{x}{e}$.

\begin{proposition}
	\label{prop:decomposepolytope}
	If $\Cut{H}{S}$ is a tight cut in a matching covered hypergraph $H$, 
	then $x \in \MatPol{H}$ if and only if there exist $x^S\in \MatPol{H_S}$, $x^{\Compl{S}}\in \MatPol{H_{\Compl{S}}}$ agreeing on $\Cut{H}{S}$ with $x = x^S\oplus x^{\Compl{S}}$.
\end{proposition}

\begin{proof}
	First, suppose $x$ lies in the perfect matching polytope of $H$.
	This means that there exist perfect matchings $M_1,\ldots, M_k$ and scalars $\lambda_1,\ldots, \lambda_k\geq 0$ with $x=\sum_{i=1}^k\lambda_i\chi^{M_i}$ and $\sum_{i=1}^k\lambda_i=1$.
	Every perfect matching $M_i$ corresponds to a pair of perfect matchings $M'_{i}, M''_{i}$ in $H_{S}$ and $H_{\Compl{S}}$ agreeing on $\Cut{H}{S}$.
	So the edge $e\in M_i\cap\Cut{H}{S}$ corresponds to the edges $e_s\in M'_{i}$ and $e_{\Compl{s}}\in M''_{i}$.
	Furthermore, since $\Cut{H}{S}$ is a tight cut it does not contain any other edge of $M_i$.
	
	Clearly,the vectors $x^S$ and $x^{\Compl{S}}$ defined by $x^S\coloneqq\sum_{i=1}^k\lambda_i\chi^{M'_{i}}$ and $x^{\Compl{S}}\coloneqq\sum_{i=1}^k\lambda_i\chi^{M''_{i}}$ agree on $\Cut{H}{S}$ and $x = x^S\oplus x^{\Compl{S}}$.
	Furthermore, $x^S$ lies in the perfect matching polytope of $H_S$, and $x^{\Compl{S}}$ in the perfect matching polytope of $H_{\Compl{S}}$.

	For the other direction, write $x^S$ and $x^{\Compl{S}}$ as a convex combination of characteristic vectors of perfect matchings: $x^S= \sum_{i=1}^k\lambda'_i\chi^{M'_{i}}$ and $x^{\Compl{S}}=\sum_{i=1}^{k'}\lambda''_i\chi^{M''_i}$.
	
	We show that $x$ lies in the perfect matching polytope of $H$ by induction on $t \coloneqq \max\Set{k,k'}$. 
	
	If $k=k'=1$, then $x^S=\chi^{M'_1}$, $x^{\Compl{S}}=\chi^{M''_1}$, and $M'_1$ and $M''_1$ agree on $\Cut{H}{S}$ as $\Fkt{x^S}{e_s} = \Fkt{x^{\Compl{S}}}{e_{\Compl{s}}}$ for all $e\in \Cut{H}{S}$.
	This means that $x=\chi^{M_1}$ where $M_1$ is the unique matching in $H$ corresponding to $M'_1, M''_1$.
	Now, assume $x$ lies in the perfect matching polytope of $H$ if $\max\Set{k,k'} \leq t$.
	
	For the induction step we can assume $k=\max\Set{k,k'} = t+1$.
	Fix some edge $e^*\in \Cut{H}{S}$ with $\Fkt{x^S}{e^*_s} > 0$, and let $I_S\coloneqq\CondSet{i}{e^*_s\in M'_i, i=1,\ldots,k}$, $I_{\Compl{S}}\coloneqq\CondSet{i}{e^*_{\Compl{s}}\in M''_i, i=1,\ldots, k'}$.
	Because of $\Fkt{x^S}{e^*_s} = \Fkt{x^{\Compl{S}}}{e^*_{\Compl{s}}}$, we know that $\sum_{i\in I_S}\lambda'_i=\sum_{i\in I_{\Compl{S}}}\lambda''_i$.
	We denote this value by $\Lambda$.
	If $\Lambda =1$, then $I_S=\{1,\ldots, k\}$, $I_{\Compl{S}}=\{1,\ldots, k'\}$, and every pair of perfect matchings $M'_i$ and $M''_j$ for $i=1,\ldots, k$, $j=1,\ldots, k'$ agrees on $\Cut{H}{S}$ which means that it corresponds to a unique perfect matching in $H$ which we denote by $M_{i,j}$.
	Now, the following procedure writes $x$ as a convex combination of perfect matchings in $H$:

	While not all $\lambda'_i=0$ choose $i\in \Set{1,\ldots, k}, j\in \Set{1,\ldots, k'}$ with $\lambda'_i>0$ and $\lambda''_j>0$, set $\mu_{i,j} \coloneqq \min\Set{\lambda'_i, \lambda''_j}$, and decrease $\lambda'_i$ and $\lambda''_j$ by $\mu_{i,j}$.

	In every step of this procedure at least one of $\lambda'_i$ or $\lambda''_j$ becomes zero, thus it terminates after a finite number of steps.
	In the end $\sum_{i,j}\mu_{i,j}\chi^{M_{i,j}} = x$ and $\sum_{i,j}\mu_{i,j} = 1$ hold.

	If $\Lambda<1$, then we consider the two vectors 
	\begin{equation*}
		y^S \coloneqq \frac{1}{1-\Lambda}\Brace{\sum_{i\in [k]\setminus I_S}\lambda'_i\chi^{M'_i}} ~\text{and}~ y^{\Compl{S}} = \frac{1}{1-\Lambda} \Brace{\sum_{i\in [k']\setminus I_{\Compl{S}}}\lambda''_i\chi^{M''_i}}.
	\end{equation*}
	The vector $y^S$ lies in the perfect matching polytope of $H_S$, $y^{\Compl{S}}$ lies in the perfect matching polytope of $H_{\Compl{S}}$, and they are written as a convex combination of less than $t+1$ characteristic vectors of perfect matchings.
	Furthermore, for every $e\in \Cut{H}{S}\setminus \{e^*\}$ we have $\Fkt{y^S}{e_s} = \Fkt{x^S}{e_s} = \Fkt{x^{\Compl{S}}}{e_{\Compl{s}}} = \Fkt{y^{\Compl S}}{e_{\Compl{s}}}$, and $\Fkt{y^S}{e^*_s} = 0 = \Fkt{y^{\Compl{S}}}{e^*_{\Compl{s}}}$.
	
	By induction hypothesis, it follows that $y = y^S\oplus y^{\Compl{S}}$ lies in the perfect matching polytope of $H$.
	On the other hand, also $z^S = \nicefrac{\Brace{\sum_{i\in I_S}\lambda'_i\chi^{M'_i}}}{\Lambda}$, and $z^{\Compl{S}} = \nicefrac{\Brace{\sum_{i\in I_{\Compl{S}}}\lambda''_i\chi^{M''_i}}}{\Lambda}$ define vectors of the perfect matching polytopes of $H_S$ and $H_{\Compl{S}}$ agreeing on $\Cut{H}{S}$ with less then $t+1$ summands in each convex combination.
	Thus, also $z \coloneqq z^S\oplus z^{\Compl{S}}$ lies in the perfect matching polytope of $H$.
	This implies that $x$ is an element of the perfect matching polytope of $H$ as $x = \Brace{1-\Lambda}y+\Lambda z$.
\end{proof}

The polytope defined by \cref{constr:nonneg} (non-negativity) and \cref{constr:degree} (degree) is called the fractional perfect matching polytope in the graph case.
In the hypergraph case, we can define this polytope in the same way and denote the fractional perfect matching polytope of a hypergraph $H$ by $\MatPolFrac{H}$.
Then, $\MatPol{H}\subseteq \MatPolFrac{H}$ with equality if and only if $\MatPolFrac{H}$ defines an integral polytope.
\cref{prop:decomposepolytope} implies that tight cut contractions preserve the integrality of the fractional perfect matching polytope.

\begin{corollary}
	\label{cor:integralpoly}
	Let $H$ be a matching covered hypergraph with a tight cut $\Cut{H}{S}$.
	If the fractional perfect matching polytope of $H$ is integral, then the fractional perfect matching polytopes of the two contractions $H_S$ and $H_{\Compl{S}}$ are integral. 
\end{corollary}

\begin{proof}
We show that the fractional matching polytope of $H_S$ is integral, the proof for $H_{\Compl{S}}$ is symmetric. 
	Let $x^S\in \MatPolFrac{H_S}$ be arbitrary.
	We will show that $x^S$ lies in the perfect matching polytope of $H_S$.
	
	For every $e\in \Cut{H}{S}$ there exists a perfect matching $M_e$ in $H_{\Compl{S}}$ containing $e_{\Compl{s}}$.
	The vector $x^{\Compl{S}}$ defined by $x^{\Compl{S}} \coloneqq \sum_{e\in \Cut{H}{S}}\Fkt{x^S}{e_s}\chi^{M_e}$ lies in the perfect matching polytope of $H_{\Compl{S}}$ as $\sum_{e\in \Cut{H}{S}}\Fkt{x^S}{e_s} = \Fkt{x^S}{\Cut{H_S}{s}} = 1$.
	In particular, $x^{\Compl{S}}$ satisfies all degree and non-negativity constraints (\cref{constr:degree,constr:nonneg}) for $H_{\Compl{S}}$. 
	Furthermore, $x^S$ and $x^{\Compl{S}}$ agree on $\Cut{H}{S}$.
	Thus, we can define $x\in \mathds{Q}^{\E{H}}$ by $x \coloneqq x^S\oplus x^{\Compl{S}}$.
	This vector is non-negative and satisfies \cref{constr:degree} on $\Fkt{V}{H}$. 
	This means that $x$ lies in the fractional perfect matching polytope of $H$ and thus in the perfect matching polytope of $H$.
	By \cref{prop:decomposepolytope}, $x^S$ lies in the perfect matching polytope of $H_S$.
\end{proof}

Our contribution to the various characterisations of balanced hypergraphs is the following observation that the notion of a balanced hypergraph is closely linked to the integrality of the fractional perfect matching polytope.

\begin{lemma}
	\label{lemma:balancedpoly}
	A hypergraph $H$ is balanced if and only if for all $S\subseteq \Fkt{V}{H}$ the fractional perfect matching polytope of the subhypergraph $\SubHG{H}{S}$ is integral. 
\end{lemma}

\begin{proof}
	Let $H$ be a balanced hypergraph, then the matching polytope is given by 
	\begin{equation*}
		\CondSet{x \in \mathds{R}^{\E{H}}}{x\geq 0,~ \Fkt{x}{\Cut{H}{v}}\leq 1},
	\end{equation*}
	see \cite{bergebook}.
	So in the balanced case the matching polytope and the fractional matching polytope coincide.
	As the perfect matching polytope is a face of the matching polytope the claim follows.
	
	On the other hand, suppose $H$ is not balanced.
	In this case we can find $S\subseteq \V{H}$ such that $\SubHG{H}{S}$ contains an odd cycle (using edges of size two) spanning $S$.
	The vector $x_e = 0.5$ for all edges of size two of this odd cycle is non-negative and satisfies \cref{constr:degree} but it cannot be written as a convex combination of incidence vectors of perfect matchings of $\SubHG{H}{S}$.
\end{proof}

Using the previous result, we prove that the tight cut contractions of \uniformable, balanced hypergraphs remain balanced.

\begin{theorem}
	\label{thm:balancedtightcutclosed}
	Let $H$ be a balanced \uniformable hypergraph and $S \subseteq \Fkt{V}{H}$ such that $\Cut{H}{S}$ is a tight cut.
	The two tight cut contractions $H_S$ and $H_{\Compl{S}}$ are balanced.
\end{theorem}

\begin{proof}
	As $H$ is \uniformable, we can assume $H$ to be $r$-uniform for some $r \in \mathds{Z}$.
	We show that $H_S$ is balanced.
	The proof for $H_{\Compl{S}}$ is symmetric (exchange $S$ with $\Compl{S}$).
	
	Due to \cref{lemma:balancedpoly} it suffices to show that for all $S'\subseteq \Fkt{V}{H_S}$ the perfect matching polytope of $\SubHG{H_S}{S'}$ is given by \cref{constr:nonneg,constr:degree}.
	So let $S'\subseteq \Fkt{V}{H_S}$.
	If $s\notin S'$, then $\SubHG{H_{S}}{S'}=\SubHG{H}{S'}$ and the perfect matching polytope of $\SubHG{H_S}{S'}$ is given by \cref{constr:nonneg,constr:degree} as $H$ is balanced.
	
	Thus, we assume $s\in S'$ and set $T \coloneqq \Brace{S'\setminus \Set{s}}\cup S$.
	Contracting $S$ in the subhypergraph $\SubHG{H}{T}$ of $H$ gives a hypergraph which is isomorphic to $\SubHG{H_{S}}{S'}$.
	Since $H$ is balanced so is $\SubHG{H}{T}$.
	Similar to the way we did in \cref{prop:decomposepolytope} we can identify vectors from $\SubHG{H}{T}$ with vectors in $\SubHG{H}{T}_S$ and $\SubHG{H}{T}_{\Compl{S}}$ and therefore with vectors in $\SubHG{H_S}{S'}$.
	
	Let $x\in \mathds{R}^{E\Brace{\SubHG{H_{S}}{S'}}}$ be a non-negative vector satisfying \cref{constr:degree} of $\SubHG{H_{S}}{S'}$.
	For every $e\in \Cut{H}{S}$ let $M_e$ be a perfect matching in $H$ containing $e$.
	Every perfect matching $M_e$ induces a perfect matching in $\SubHG{H}{T}$ containing $e\cap T$.
	Set $N_e \coloneqq \CondSet{f\in M_e}{f\subseteq S} \cup \Set{e\cap T}$ and look at the vector $z\in \mathds{R}^{E\Brace{\SubHG{H}{T}}}$ defined by $\Fkt{z}{f} \coloneqq \Fkt{x}{f}$ for $f \in \Fkt{E}{\SubHG{H_{S}}{S'}} \cap \Fkt{E}{\SubHG{H}{T}}$.
	Add $\Fkt{x}{e_s \cap S'} \chi^{N_e}$ to $z$ for every $e \in \Cut{H}{S}$.
	Clearly, the resulting $z$ is non-negative.
	
	We show that $z$ satisfies \cref{constr:degree} for every $v \in T$.
	First, let $v \in S' \setminus \Set{s} = T \cap \Compl{S}$.
	An edge in $\SubHG{H_{S}}{S'}$ containing $v$ lies either completely in $\Compl{S}$ or it contains $s$.
	In the first case, $e\in \Fkt{E}{\SubHG{H}{T}}$ and $\Fkt{z}{e} = \Fkt{x}{e}$.
	In the second case, the edge is of the form $e_s \cap S'$ for some $e \in \Cut{H}{S}$ and $\Fkt{z}{e \cap T} = \Fkt{x}{e_s \cap S'}$.
	
	In total we get that $\Fkt{z}{\Cut{}{v}} = \Fkt{x}{\Cut{}{v}} = 1$ for all $v \in S' \setminus \Set{s}$.
	Now, consider a vertex $v \in S$.
	This vertex is incident to exactly one edge in $N_e$ for every $e\in \Cut{H}{S}$, and it is not incident to any edge $f \in \Fkt{E}{\SubHG{H_{S}}{S'}}\cap \Fkt{E}{\SubHG{H}{T}}$.
	Thus, we get 
	\begin{align*}
		\Fkt{z}{\Cut{}{v}} = \sum_{e \in \Cut{H}{S}} \Fkt{x}{e_s \cap S'} = \sum_{\substack{e \in E\Brace{\SubHG{H_{S}}{S'}},\\ s \in e}} \Fkt{x}{e} = \Fkt{x}{\Cut{\SubHG{H_{S}}{S'}}{s}} = 1.
	\end{align*}

	As $H$ is balanced, it follows that $z$ lies in the perfect matching polytope of $\SubHG{H}{T}$.
	Thus, we can write $z$ as a convex combination of incidence vectors of perfect matchings: $z = \sum_{i=1}^k \lambda_i \chi^{M_i}$, where $M_i$ is a perfect matching of $\SubHG{H}{T}$.
	For every $i = 1,\ldots, k$ let
	\begin{align*}
		M'_i\coloneqq\CondSet{f\in M_i}{f\subseteq S' \setminus \Set{s}} \cup \CondSet{e_s \cap S'}{e\in \Cut{H}{S}, e \cap T \in M_i}.
	\end{align*}
	As $\Fkt{x}{e} = \Fkt{z}{e}$ for all $e \in \Fkt{E}{\SubHG{H_{S}}{S'}}\cap \Fkt{E}{\SubHG{H}{T}}$ and $\Fkt{x}{e_s \cap S')} = \Fkt{z}{e \cap T}$ for all $e \in \Cut{H}{S}$, we get $x = \sum_{i=1}^k \lambda_i \chi^{M'_i}$.
	We have to show that each $M'_i$ is a perfect matching in $\SubHG{H_{S}}{S'}$.
	It is clear that every $v \in S'\setminus \Set{s}$ is covered exactly once by each $M_i$.
	Furthermore, the vertex $s$ is contained in at least one of the edges of $M_i$.
	Otherwise, no $e \in \Cut{H}{S}$ with $e \cap T \in M_i$ exists.
	This implies that $e \subseteq S$ or $e \subseteq \Compl{S}$ for $e \cap T\in M_i$, and the edges $e\subseteq S$ of $M_i$ form a matching covering $S$.
	But then $\Abs{S}$ is divisible by $r$ which is impossible by \cref{obs:tightcut} because $S$ defines a tight cut and $H$ is uniform.
	
	It remains to show that $s$ is covered by exactly one edge of $M'_i$.
	We get
	\begin{itemize}
		\item[] $1 = \sum_{e\in \Cut{H}{S}}\Fkt{x}{e_s \cap S'} = \sum_{i=1}^k \lambda_i \Deg{\PartHG{H_S}{M'_i}}{s} \geq \sum_{i=1}^k \lambda_i = 1$.
	\end{itemize}
	This implies $\Deg{\PartHG{H_S}{M'_i}}{s}=1$ and thus $M'_i$ is a perfect matching.
	Thus, $x = \sum_{i=1}^k \lambda_i \chi^{M'_i}$ is a convex combination of perfect matchings, so $x$ lies in the perfect matching polytope of $\SubHG{H_{S}}{S'}$. This concludes the proof.
\end{proof}

We have seen that if $\Cut{H}{S}$ is a tight cut in a hypergraph $H$, then $H$ is matching covered if and only if $H_S$ and $H_{\Compl{S}}$ are matching covered. In the graph case there is a larger class of cuts with this property, so-called separating cuts, which we investigate in the remainder of this section.
\begin{definition}
A cut $\Cut{H}{S}$ in a matching covered hypergraph $H$ is called \emph{separating}\index{separating cut} if $H_S$ and $H_{\Compl{S}}$ are matching covered. 
\end{definition}
\Cref{prop:matcoveredtight} implies that every tight cut is separating. If a graph has a non-tight separating cut, then the non-negativity and degree constraints do not suffice to describe the perfect matching polytope or in other words the fractional perfect matching polytope of this graph is not integral. 

If $\Cut{H}{S}$ is a separating cut in a hypergraph $H$, then the inequality $x(\Cut{H}{S})\geq 1$ might not be valid for the perfect matching polytope. However, if $H$ is \uniformable, then by similar arguments as in \cref{obs:tightcut} $|M\cap \Cut{H}{S}|\geq 1$ holds for all perfect matchings $M$. Thus, $x(\Cut{H}{S})\geq 1$ is a valid inequality for all vectors $x\in \MatPol{H}$. As in the graph case, we use a separating cut that is not tight in a \uniformable hypergraph to construct a non-negative vector that satisfies all degree constraints but does not lie in the perfect matching polytope.
\begin{theorem}\label{thm:uniformseparating}
	If $H$ is a matching covered, \uniformable hypergraph with a separating cut $\Cut{H}{S}$ that is not tight, then the fractional perfect matching polytope is not integral.
\end{theorem}
\begin{proof}
	Suppose that $H$ has a separating cut $\Cut{H}{S}$ that is not tight. We proceed as in the graph case~\cite{decarvalho2004} by constructing a vector $x$ lying in the fractional perfect matching polytope with $x(\Cut{H}{S})<1$.
	
	Let $M_0$ be a perfect matching of $H$ with $\Abs{M_0\cap \Cut{H}{S}}\geq 2$. For every $e\in M_0$ let 
	$M_e$ be a perfect matching containing $e$ and intersecting $\Cut{H}{S}$ in exactly one hyperedge, namely $e$. 
	These matchings exists because $\Cut{H}{S}$ is assumed to be a separating cut that is not tight.
	The vector $x\in \mathds Q^{\Fkt{E}{H}}$ defined by
		\begin{align*}
		x = \frac{1}{\Abs{M_0}-1}\left(\sum_{e\in M_0}\chi^{M_e}-\chi^{M_0}\right)
		\end{align*}
	is non-negative and satisfies $x(\delta_H(v))=1$ for all $v\in \Fkt{V}{H}$ but 
	\[x(\Cut{H}{S}) = 
	\frac{\Abs{M_0}-\Abs{M_0\cap \Cut{H}{S}}}{\Abs{M_0}-1}<1.\]
	Thus, $x\notin \MatPol{H}$, which 
	implies 
	that the fractional perfect matching polytope of $H$ is not integral.
\end{proof}
The previous theorem implies that \uniformable hypergraphs for which the non-negativity and degree constraints describe an integral polytope have no non-tight separating cuts. This holds in particular for unimodular, balanced, normal, and mengerian hypergraphs.
\begin{corollary}\label{cor:uniformseparating}
	If $H$ is a \uniformable matching covered hypergraph with an integral fractional perfect matching polytope, then every separating cut of $H$ is tight.
\end{corollary}
If we consider hypergraphs that are not \uniformable, then it is possible that the fractional perfect matching polytope is integral but the hypergraph has a separating cut that is not tight. For example, if we take $H$ to be the complete bipartite graph $K_{2,2}$ together with singleton hyperedges $\{v\}$ for every vertex $v$, then we obtain a unimodular hypergraph with a non-tight separating cut, see \cref{fig:sepcut}.

\begin{figure}
\centering

\begin{tikzpicture}
\tikzset{loop/.style = {shape = circle, draw, inner sep = 4pt, line width = 2pt, DarkTangerine}}

	\node[vertex] (a1) at (0,0){};
	\node[vertex] (a2) at ($(a1)+(0,2)$){};
	
	\node[vertex] (b1) at ($(a1)+(2,0)$){};
	\node[vertex] (b2) at ($(b1)+(0,2)$){};
	
	\node[loop] (l1) at ($(a1)$){};
	\node[loop] (l2) at ($(a2)$){};
	\node[loop] (l3) at ($(b1)$){};
	\node[loop] (l4) at ($(b2)$){};
	\draw[Hedge] (a1) to (b1);
	\draw[Hedge] (a1) to (b2);
	\draw[Hedge] (a2) to (b1);
	\draw[Hedge] (a2) to (b2);

	\coordinate (Sstart) at ($(a1)+(0.8,-0.5)$);
	\coordinate (Send) at ($(a2)+(0.8,0.5)$);
	\def\CutSCurve{%
			(Sstart)  ..(1,1).. (Send)
		}
	\draw[BostonUniversityRed,line width=3pt,opacity=0.7,use Hobby shortcut, closed=false] \CutSCurve;
	\node[BostonUniversityRed] at ($(a1)+(0.5,-0.3)$){$S$};
	\node[BostonUniversityRed] at ($(a1)+(1.2,-0.3)$){$\bar S$};
\end{tikzpicture}

\caption{A unimodular hypergraph with a separating cut that is not tight.}
\label{fig:sepcut}
\end{figure}
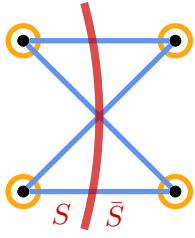

\Cref{cor:uniformseparating} generalizes the fact that a bipartite graph has no non-tight separating cut, which can also be proven without polyhedral methods, see for example \cite{carvalho2002}. Namely, if $\Cut{G}{S}$ is a non-tight separating cut in a graph $G$, then the subgraphs $G[S]$ and $G[\Compl{S}]$ are non-bipartite. For hypergraphs we show that the shores of a non-tight separating cut cannot induce $r$-partite subhypergraphs.
\begin{theorem}
Let $H$ be a matching covered, $r$-uniform hypergraph and $S\subseteq \Fkt{V}{H}$ be a set of vertices such that $\Cut{H}{S}$ is a non-tight separating cut. The subhypergraphs $H[S]$ and $H[\Compl{S}]$ of $H$ induced by $S$ and $\Compl{S}$ are not $r$-partite.
\end{theorem}
\begin{proof}
We only show that $H[S]$ is not $r$-partite. The proof for $H[\Compl{S}]$ is similar.

Suppose that there exists a partition $S_1\cup \ldots \cup S_r$ of $S$ into $r$ sets such that $\Abs{e\cap S_1}=\ldots = \Abs{e\cap S_r}=1$ for all $e\in \Fkt{E}{H[S]}$.
We choose any hyperedge $f\in \Cut{H}{S}$ and let $M_f$ be a perfect matching of $H$ with $M_f\cap \Cut{H}{S}=\{f\}$.
The set $M'_f \coloneqq \CondSet{e\in M_f}{e\subseteq S}$ is a matching of $H[S]$ covering $S\setminus f$.
Without loss of generality we assume that there exists an index $k$ with $1\leq k\leq r-1$ such that $f$ intersects $S_1,\ldots, S_k$ and $f$ has an empty intersection with $S_{k+1},\ldots, S_r$.
It follows that $\Abs{S_1\setminus f}=\ldots =\Abs{S_k\setminus f}=\Abs{S_{k+1}}=\ldots = \Abs{S_r}$.
Furthermore, as we have chosen $f$ arbitrarily we get that $\Abs{S_1\setminus e}=\ldots = \Abs{S_k\setminus e}=\Abs{S_{k+1}} = \ldots =\Abs{S_r}$ for every $e\in \Cut{H}{S}$.
Now, let $M$ be a perfect matching of $H$ intersecting $\Cut{H}{S}$ in more than one edge.
Again, the set $M' \coloneqq \CondSet{e\in M}{e\subseteq S}$ forms a matching of $H[S]$.
If $\Set{m_1,\ldots, m_s} =  M'\cap \Cut{H}{S}$, then $M'$ covers the vertex set $S\setminus \left(m_1\cup \ldots \cup m_s\right)$.
This implies that $\Abs{S_1\setminus \left(m_1\cup \ldots \cup m_s\right)}=\ldots = \Abs{S_r\setminus \left(m_1\cup \ldots \cup m_s\right)}$, which is impossible as $\Abs{S_1\setminus \left(m_1\cup \ldots \cup m_s\right)} = \Abs{S_1} - s < \Abs{S_r} = \Abs{S_r\setminus \left(m_1\cup \ldots \cup m_s\right)}$.
Thus, $H[S]$ is not $r$-partite.
\end{proof}
As a corollary, we directly obtain that an $r$-partite hypergraph cannot have a non-tight separating cut. 
\begin{corollary}
If $H$ is an $r$-partite hypergraph, then every separating cut is tight.
\end{corollary}
Carvalho, Lucchesi, and Murty proved that the reverse implication of \cref{thm:uniformseparating} holds in the graph case. Namely, every graph with a non-integral fractional matching polytope contains a separating cut that is not tight. This is not true for hypergraphs of rank at least three. No $3$-partite hypergraph has a non-tight separating cut but there are $3$-partite hypergraphs with a non-integral fractional perfect matching polytope, for example a complete $3$-partite hypergraph.

\section{Algorithmic Consequences}\label{sec:algo}

Finding a perfect matching of maximum weight in a hypergraph with some weights on the edges is an \NP-complete problem.
In practice, it can be solved quite efficiently within a branch-and-cut framework by an integer programming solver.
However, the performance heavily depends on the size of the hypergraph.
Therefore, it is of great use if one can decompose the problem into smaller parts.
This can be done using the tight cut decomposition procedure as we have seen in the previous section.

Namely, every tight cut decomposition yields a decomposition of the perfect matching polytope.
By the uniqueness result of \cref{thm:unique}, all tight cut decompositions of a \uniformable hypergraph are equally good.
Thus, in order to determine a tight cut decomposition we only need an algorithm that finds a non-trivial tight cut or concludes that none exists.
We do not know whether a polynomial time algorithm for this problem exists.
In graphs one can determine a tight cut decomposition in polynomial time by computing the brick decomposition (a special kind of tight cut decomposition) of a graph, this is described in the second section of a result by Edmonds et al.~\cite{edmonds1982}.
For hypergraphs it is not clear how to generalise this result.

In this section we give a polynomial time algorithm that finds a non-trivial tight cut in a uniform balanced hypergraph, or concludes that none exists. In a uniform balanced hypergraph every separating cut is tight by \cref{cor:uniformseparating}. We can exploit this fact as follows: 

For every $e\in \Fkt{E}{H}$ we compute a perfect matching $M_e$ containing $e$.
This can be done in polynomial time by linear programming as the fractional perfect matching polytope of a balanced hypergraph is integral.
Now, if $\Cut{H}{S}$ is a tight cut, then $\Abs{M_e\cap \Cut{H}{S}}=1$ for all $e\in \Fkt{E}{H}$.
On the other hand, if $\Cut{H}{S}$ is a cut such that $\Abs{M_e\cap \Cut{H}{S}}=1$ for all $e\in \Fkt{E}{H}$, then $\Cut{H}{S}$ is a separating cut and also a tight cut as $H$ is a uniform balanced hypergraph.
In total, $\Cut{H}{S}$ is tight if and only if $\Abs{M_e\cap \Cut{H}{S}}=1$ for all $e\in \Fkt{E}{H}$, where $\CondSet{M_e}{e\in \Fkt{E}{H}}$ is an arbitrary set of perfect matchings such that $M_e$ contains $e$ for every $e\in \Fkt{E}{H}$.  

For every $e\in \Fkt{E}{H}$ we choose a perfect matching $M_e$ such that $e\in M_e$, and define a weight function $w:\Fkt{E}{H}\rightarrow \mathds Z$ by $w(f) \coloneqq \Abs{\CondSet{e\in \Fkt{E}{H}}{f\in M_e}}$ for every $f\in \Fkt{E}{H}$.
The value $w(f)$ of the function $w$ at a hyperedge $f$ is equal to the number of perfect matchings in the set $\CondSet{M_e}{e\in \Fkt{E}{H}}$ containing $f$.
If $S\subseteq \Fkt{V}{H}$ is such that its size is not divisible by $r$, then $\Abs{M\cap \Cut{H}{S}}\geq 1$ for all perfect matchings $M$ of $H$.
In particular, $\Abs{M_e\cap \Cut{H}{S}}=1$ for all $e\in \Fkt{E}{H}$ if and only if $w(\Cut{H}{S})=\Abs{\Fkt{E}{H}}$ holds.
Thus, $S\subseteq \Fkt{V}{H}$ defines a tight cut if and only if $\Abs{S}\not \equiv_r 0$ and $w(\Cut{H}{S})=\Abs{\Fkt{E}{H}}$.
By Corollary 10.4.7~\cite{GroetschelLovaszSchrijver1988}, we can solve $\min\CondSet{w(\Cut{H}{S})}{S\subseteq \Fkt{V}{H}, \Abs{S}\not \equiv_r 0}$ in polynomial time.
However, we want to find a non-trivial tight cut, thus we have to demand that $\Abs{S}\geq 2$ and $\Abs{S}\leq |\Fkt{V}{H}|-2$. 
We show that also the optimization problem
\begin{align}
	\min\CondSet{w(\Cut{H}{S})}{S\subseteq \Fkt{V}{H}, 2\leq \Abs{S}\leq |\Fkt{V}{H}|-2, \Abs{S}\not \equiv_r 0}\label{non-trivTightCut}
\end{align}
is polynomial-time solvable.

Therefore, let $A,B\subseteq \Fkt{V}{H}$ be disjoint sets of vertices.
The family of sets $\mathcal C(A,B)\coloneqq\{S\subseteq \Fkt{V}{H} \mid A\subseteq S\subseteq \Fkt{V}{H}\setminus B\}$ has the property that for every $S,T\in \mathcal C(A,B)$ also $S\cap T$ and $S\cup T$ lie in $\mathcal C(A,B)$.
Such a family is called a lattice family~\cite{GroetschelLovaszSchrijver1988}.
Again, by Corollary 10.4.7~\cite{GroetschelLovaszSchrijver1988} applied to $\mathcal C(A,B)$, we can solve $\min\CondSet{w(\Cut{H}{S})}{S\in \mathcal C(A,B), \Abs{S}\not \equiv_r 0}$ in polynomial time for every pair of fixed sets $A,B\subseteq \Fkt{V}{H}$. 

Now, problem \cref{non-trivTightCut} can be solved by calculating for all disjoint subsets $A,B$ of $\Fkt{V}{H}$ with $\Abs{A}=\Abs{B}=2$ an optimal solution $S_{A,B}$ to the optimization problem $\min\{w(\Cut{H}{S}) \mid S\in \mathcal C(A,B), \Abs{S}\not \equiv_r 0\}$, and choosing a set
\begin{align*}
	S^* \coloneqq \text{argmin}\CondSet{w(\delta_H(S_{A,B}))}{A,B\subseteq \Fkt{V}{H}, A\cap B= \emptyset, \Abs{A}=\Abs{B}=2}.
\end{align*}

 If $w(S^*)=\Abs{\Fkt{E}{H}}$, then $\delta_H(S^*)$ is a non-trivial tight cut in $H$.
 Otherwise, $w(S^*)>\Abs{\Fkt{E}{H}}$, and $H$ contains only trivial tight cuts.
 As there are $\mathcal O(\Abs{\Fkt{V}{H}}^4)$ subsets with $A,B\subseteq \Fkt{V}{H}$ with $\Abs{A}=\Abs{B}=2$, this algorithm runs in polynomial time.
 In total, we get the following result.
\begin{theorem}
Let $H$ be an $r$-uniform, matching covered, balanced hypergraph.
There exists a polynomial time algorithm that either outputs a non-trivial tight cut $\Cut{H}{S}$ or concludes that $H$ has only trivial tight cuts.
\end{theorem}

\section{Conclusion}


We have shown that the tight cut decomposition is unique for \uniformable hypergraphs and refuted a similar result for general hypergraphs.
We also established that balanced \uniformable hypergraphs are closed under tight cut contractions.
This might allow us to further generalise the theory of braces, bipartite matching covered graphs without non-trivial tight cuts, to the (balanced) hypergraphic setting.

We also made a first step towards a computational approach on the tight cut decomposition by providing a polynomial algorithm that finds a non trivial tight cut in a uniformable and balanced hypergraph if one exists.

Two questions are immediate consequences of the uniqueness result:

\begin{enumerate}
	\item Is there a structural characterisation of hyperbricks?
	
	\item Is there a polynomial time algorithm to find a non-trivial tight cut in a \uniformable matching covered hypergraph?
\end{enumerate}

To the graphic versions of these questions the answers are well known to be yes.
The results answering these questions on graphs are the cornerstones of the theory of matching covered graphs.
So, finding appropriate answers in the world of \uniformable hypergraphs, one might be able to establish a similar theory in a more general setting.

We thank Matthias Walter for pointing out an error in an earlier version of this paper.

\bibliographystyle{splncs03}
\bibliography{literature}

\appendix

\end{document}